%% file: arxiv.tex
\documentclass[10pt, a4paper]{article}

\usepackage{graphicx}
\usepackage{hyperref}
\usepackage{fullpage}
\usepackage{xcolor}
\usepackage{amssymb}
\usepackage{amsmath}
\usepackage{amsthm}
\usepackage{MnSymbol}
\usepackage{multirow}
\usepackage[capitalise,noabbrev]{cleveref}

\usepackage{algorithm}
\usepackage{algpseudocode}

\usepackage{subfig}

\providecommand{\keywords}[1]{\textbf{Key words. } #1}
\providecommand{\amssubj}[1]{\textbf{AMS subject classification. } #1}

\newcommand{\email}[1]{Email: #1}

\DeclareMathOperator*{\diago}{diag}

\newtheoremstyle{mystyle}
  {}
  {}
  {\itshape}
  {}
  {\bfseries}
  {.}
  { }
  {}

\theoremstyle{mystyle}
\newtheorem{definition}{Definition}
\newtheorem{example}{Example}
\newtheorem{remark}{Remark}
\newtheorem{corollary}{Corollary}

\definecolor{green_plot1}{rgb}{.4,.7,0.2}
\definecolor{blue_plot1}{rgb}{0,0,1}
\definecolor{red_plot1}{rgb}{1,0,0}
\definecolor{cyan_plot1}{rgb}{0,0.8,0.8}

\title{A Low-rank Method for Parameter-dependent Fluid-structure Interaction Discretizations with Hyperelasticity}

\author{Peter~Benner\thanks{Corresponding author}
    \thanks{Max Planck Institute for
Dynamics of Complex Technical Systems, Magdeburg and Otto von Guericke
University Magdeburg, Germany, (\email{benner@mpi-magdeburg.mpg.de})}
\and Thomas~Richter\thanks{Otto von Guericke University Magdeburg,
  Germany} \and 
Roman~Weinhandl\thanks{This work was mainly completed while
    this author was with Otto von Guericke University Magdeburg and
    Max Planck Institute for Dynamics of Complex Technical Systems,
    Magdeburg, Germany} }

\begin{document}

\date{}
\maketitle

\begin{abstract}
Fluid-structure interaction models are used to study how a material interacts with different fluids at different Reynolds numbers. Examining the same model not only for different fluids but also for different solids allows to optimize the choice of materials for construction even better. A possible answer to this demand is parameter-dependent discretization. Furthermore, low-rank techniques can reduce the complexity needed to compute approximations to parameter-dependent fluid-structure interaction discretizations.

Low-rank methods have been applied to parameter-dependent linear fluid-structure interaction discretizations. The linearity of the operators involved allows to translate the resulting equations to a single matrix equation. The solution is approximated by a low-rank method. In this paper, we propose a new method that extends this framework to nonlinear parameter-dependent fluid-structure interaction problems by means of the Newton iteration. The parameter set is split into disjoint subsets. On each subset, the Newton approximation of the problem related to the upper median parameter is computed and serves as initial guess for one Newton step on the whole subset. This Newton step yields a matrix equation whose solution can be approximated by a low-rank method.

The resulting method requires a smaller number of Newton steps if
compared with a direct approach that applies the Newton iteration to
the separate problems consecutively. In the experiments considered,
the proposed method allows to compute a low-rank approximation up to
twenty times faster than by the direct approach.
\end{abstract}

 \vspace{.4cm}
\par
\amssubj{15A69, 49M15, 65M22, 74F10}



\section{Introduction}\label{section_introduction1}

\input{introduction.tex}

\section{The stationary nonlinear fluid-structure interaction problem}\label{chap_weak_form1}
Let $\Omega$, $F$, $S \subset \mathbb{R}^d$ be open for $d \in \{2,3\}$ such that $\bar{F}  \cup \bar{S}=\bar\Omega$ and $S \cap F = \emptyset$. We are interested in an FSI problem that uses the stationary Navier-Stokes equations \cite[Section 2.4.5.3]{Ric17} to model the fluid part $F$. On the solid part $S$, we use the stationary model of a Saint Venant-Kirchhoff material \cite[Definition 2.18]{Ric17}. By $\Gamma_{\text{int}}=\partial F \cap \partial S$, we denote the interface. The outer boundary $\partial\Omega$ is split into the Dirichlet part $\Gamma^F\subset \partial\Omega$ and the fluid outflow part $\Gamma_f^{\text{out}}\subset \partial F \cap \partial\Omega$, where a Neumann outflow conditions of do-nothing type holds~\cite{Heywood1996}. We denote the $\mathcal{L}^2$ scalar product on $F$ and $S$ by $\langle \cdot, \cdot \rangle_F$ and $\langle \cdot , \cdot \rangle_S$, respectively. The weak formulation of the coupled nonlinear FSI problem with a vanishing right hand side is given by
\begin{equation}
\label{equation_fsi_problem_weak1}
\begin{aligned}
\langle \nabla \cdot v, \xi  \rangle_F&=0\text{,}\\
\mu_s \langle \nabla u + \nabla u^T+\nabla u^T \nabla u, \nabla \varphi \rangle_S +\lambda_s \langle \operatorname{tr}(\nabla u+\frac{1}{2}\nabla u^T \nabla u)I, \nabla \varphi \rangle_S\quad\\
+\rho_f \langle (v \cdot \nabla)v, \varphi  \rangle_F + \nu_f \rho_f \langle \nabla v+\nabla v^T, \nabla \varphi \rangle_F- \langle p, \nabla \cdot \varphi \rangle_F &=0 \quad \text{and}\\
\langle \nabla u, \nabla \psi  \rangle_F&=0\text{,}
\end{aligned}
\end{equation}
where $p \in L^2(F)$ denotes the pressure, $v \in v_{\text{in}} + H_0^1(F, \Gamma_f^D\cup\Gamma_\text{int})^d$ the velocity and  $u \in H_0^1(\Omega)^d$ the deformation. By $v_{\text{in}}\in H^1(F)^d$ we introduce an extension of the Dirichlet data on $\Gamma^D$ into the fluid domain. For the test functions of the divergence equation, momentum and deformation equation we assume $\xi \in L^2(F)$, $\varphi \in H_0^1(\Omega, \Gamma^D)^d$ and $\psi \in H_0^1(F)^d$, respectively. $H_0^1(\Omega,\Gamma^D)^d$ denotes the $L^2(\Omega)^d$ functions that are weakly differentiable and have a trace equal to zero on $\Gamma^D$. The weakly differentiable functions  $H_0^1(\Omega)^d$ have a trace equal to zero on $\partial \Omega$. The fluid parameters involved are the fluid density $\rho_f \in \mathbb{R}$ and the kinematic fluid density $\nu_f \in \mathbb{R}$. The Saint Venant-Kirchhoff model equations involve the first Lam\'e parameter $\lambda_s \in \mathbb{R}$ and the shear modulus $\mu_s \in \mathbb{R}$. The interface coupling conditions at $\Gamma_\text{int}=\partial S\cap \partial F$ are implicitly included in~\eqref{equation_fsi_problem_weak1}. The geometric condition $u|_F=u|_S$ on $\Gamma_{\text{int}}$ follows from the global approach $u\in H^1_0(\Omega)^d$ and likewise, the kinematic condition $v=0$ on $\Gamma_\text{int}$ is embedded into the trial space. The dynamic condition
\[
\Big(\mu_s\big(\nabla u+\nabla u^T+\nabla u^T\nabla u\big)+\lambda_s\operatorname{tr}\big(\nabla u+\frac{1}{2}\nabla u^T\nabla u\big)I\Big)n_S+
\Big(\nu_f\rho_f\big(\nabla v+\nabla v^T\big)-pI\Big)n_F=0
\]
follows with integration by parts as the common test function $\phi$ is globally defined. We refer to~\cite[Sec. 3.4]{Ric17} for details on the derivation of the variational formulation. By $n_S$ and $n_F$ we denote outward facing unit normal vectors on $\partial S$ and $\partial F$, respectively.

\subsection{Finite element formulation and Newton iteration}

We will consider monolithic finite element discretizations of the coupled fluid-structure interaction problem~\eqref{equation_fsi_problem_weak1}. To keep the notation simple we assume that velocity $v_h\in V_h\subset H_0^1(\Omega,\Gamma_f^D)^d)$, deformation  $u_h\in W_h\subset H^1_0(\Omega)^d$ and pressure $p_h\in L_h\subset L^2(F)$ are conforming finite element spaces, where the pair $V_h\times L_h$ satisfies the inf-sup condition on $F$. We refer to the literature for details and peculiarities on finite element discretizations of FSI problems~\cite{Donea1982,RichterWick2010,BazilevsTakizawaTezduyar2013,Ric17}. In the following we will just skip the index $h$ referring to the spatial discretization. The arising nonlinear algebraic problems will be solved by the Newton iteration, which usually shows optimal convergence for fluid-structure interactions~\cite{HronTurek2006a,FailerRichter2019}. 

The solution components are combined to one single variable $x=(v_h,u_h,p_h)\in X_h:=V_h\times W_h\times L_h$ with $M:=\operatorname{dim}(X_h)$ and the Newton iteration $j-1\mapsto j$ is formally given by
\begin{equation}\label{equation_intro_linsys1}
  \begin{aligned}
    \Big( A_0+\mu_sA_1+\rho_f\nu_f A_2  +\rho_fJ_{\rho}(x_{j-1})+\mu_s J_\mu(x_{j-1})+\lambda_s J_\lambda(x_{j-1}) \Big)\delta x_j&=b_D-g(x_{j-1}),\\
     x_j &= x_{j-1}+\delta x_j.
  \end{aligned}
\end{equation}
The matrices $A_0$, $A_1$, $A_2 \in \mathbb{R}^{M \times M}$ are discrete linear operators, $J_\rho(\cdot)$, $J_\mu(\cdot)$, $J_\lambda(\cdot) \in \mathbb{R}^{M \times M}$ are discrete Jacobian matrices of nonlinear operators evaluated at the argument given.
$g(\cdot)$ evaluates the variational residual at the argument. The Jacobian matrices $J_\rho(\cdot)$, $J_\mu(\cdot)$ and $J_\lambda(\cdot)$ depend on $x_{j-1}$, the approximation of the previous linearization step.
Solving the linear systems is challenging due to the bad conditioning of the Jacobian and the missing structure, the matrix is neither symmetric nor positive definite. The coupled system includes the Navier-Stokes equation and is hence of a saddle-point type. Linear solvers are well investigated~\cite{Gee2010,Richter2015,Deparis2016,Aulisa2018,FailerRichter2019} but in two dimensions, direct solvers are also applicable.

\subsection{Parameter dependent formulation}

Since the behavior of an FSI model varies if parameters such as the shear modulus or the kinematic fluid viscosity changes, we are interested in a parameter-dependent discretization with respect to $m \in \mathbb{N}$ different parameter combinations $(\rho_f^i,\nu_f^i,\lambda_s^i,\mu_s^i)$ for $i\in \{1,\dots,m\}$. Here we will consider variations of the shear moduli $\mu_s^i$ and the kinematic fluid viscosity $\nu_f^i$. Hence, the nonlinear problem~\eqref{equation_intro_linsys1} is to be solved for $m$ solutions $x^1,x^2,\dots,x^m\in X_h$ each of them being $M$-dimensional. For a given parameter choice $i \in \{1,...,m\}$, at the Newton step $j \in \mathbb{N}$, the Jacobians $J_\rho(\cdot)$, $J_\mu(\cdot)$ and $J_\lambda(\cdot)$ depend on $x^i_{j-1}$, the approximation of the previous linearization step, and, therefore, on the parameter index $i$.  

Due to this nonlinear dependency, the parameter dependent equation can not be translated to a matrix equation similar to the linear case discussed in \cite{WeiBR20}. Instead, what we propose in this paper is to split the parameter set into $K \in \mathbb{N}$ disjoint subsets:
\begin{align*}
S^{\mu,\nu}:=\{(\mu_s^i, \nu_f^i)\}_{i \in \{1,...,m\}} = \bigcupdot \limits_{k=1}^K \mathcal{I}_k \text{.}
\end{align*}
The subsets have to be chosen such that the problems that are clustered within one subset $\mathcal{I}_k$ do not differ too much from each other. Details regarding the clustering will be discussed later. For each subset $\mathcal{I}_k$, $\tilde{x}^k_{\epsilon_N}\in \mathbb{R}^M$, the Newton approximation of the problem related to the upper median index, is computed up to a given accuracy $\epsilon_N>0$. $\tilde{x}^k_{\epsilon_N}$ is then used as initial guess for one Newton step on the whole subset $\mathcal{I}_k$. This translates to a Newton step that can be written as the following matrix equation. For all $k \in \{1,...,K\}$, find $S_k \in \mathbb{R}^{M \times |\mathcal{I}_k|}$ such that
\begin{multline}\label{equation_intro_mat1}
  A_0 S_k+A_1S_kD_\mu^k+\rho_f A_2 S_k D_\nu^k + \rho_f J_\rho(\tilde{x}^k_{\epsilon_N})S_k+J_\mu(\tilde{x}^k_{\epsilon_N})S_k D_\mu^k+\lambda_s J_\lambda(\tilde{x}^k_{\epsilon_N})S_k\\
  =\big(b_D-g(\tilde{x}^k_{\epsilon_N},0,0)\big) \otimes (1,...,1)-\big(A_1 \tilde{x}_{\epsilon_N}^k +g_\mu(\tilde{x}^k_{\epsilon_N})\big)\otimes \diago(D_\mu^k)^T
-\rho_f A_2 \tilde{x}^k_{\epsilon_N} \otimes \diago(D_\nu^k)^T,
\end{multline}
where 
\begin{align}\label{equation_diag_intro1}
D_\mu^k:=\diago \limits_{(\mu_s^i,\nu_f^i) \in \mathcal{I}_k}(\mu_s^i) \text{,} \qquad D_\nu^k:=\diago \limits_{(\mu_s^i, \nu_f^i) \in \mathcal{I}_k}(\nu_f^i) \in \mathbb{R}^{|\mathcal{I}_k|\times|\mathcal{I}_k|} \qquad \text{and}
\end{align}
$g_\mu(\cdot)$ evaluates the residual of the operator that is, in the FSI problem, multiplied with the shear modulus. $S_k$ in (\ref{equation_intro_mat1}) is the Newton update for all problems related to the parameter set $\mathcal{I}_k$ and can be approximated by a low-rank method such as the GMRESTR or the ChebyshevT method from \cite[Algorithm 2, Algorithm 3]{WeiBR20}. The low-rank approximation $\hat{S}_k$ to (\ref{equation_intro_mat1}) is then added to the initial guess on $\mathcal{I}_k$. We obtain the Newton approximation
\begin{align*}
\hat{X}_k:=\tilde{x}^k_{\epsilon_N}\otimes (1,...,1)+\hat{S}_k \qquad \text{for all } \qquad k \in \{1,...,K\} \text{.}
\end{align*}
The global approximation to the parameter-dependent problem is
\begin{align*}
\hat{X}:=[\hat{X}_1 | ... | \hat{X}_K] \text{.}
\end{align*}
If $\hat{S}_k$ has rank $R_k$ for all $k \in \{1,...,K\}$, the rank of $\hat{X}$ is at most $K+\sum \limits_{k=1}^K R_k$.

To obtain $\tilde{x}^k_{\epsilon_N}$, multiple Newton steps of the standard FSI problem~\eqref{equation_intro_linsys1} are needed. All $K$ of them can be computed in parallel. On the subset itself, we perform only one Newton step for the matrix equation.

\section{Discretization and solution of the matrix-Newton iteration}\label{chap_pardep_newton1}
For the sake of readability, we restrict to the case of a parameter-dependent discretization with respect to two parameters. In \cref{chap_extension_further1}, we explain how to extend the resulting method to further parameters. The parameters of interest are the $m_1 \in \mathbb{N}$ shear moduli
and the $m_2 \in \mathbb{N}$ kinematic fluid viscosities
\begin{equation}
\{\mu_s^{i_1}\}_{i_1 \in \{1,...,m_1\}} \subset \mathbb{R},\quad
\{ \nu_f^{i_2} \}_{i_2 \in \{1,...,m_2\}} \subset \mathbb{R}.
\end{equation}
With this choice we face, in total, $m=m_1m_2$ different FSI problems.

Consider a finite element discretization of (\ref{equation_fsi_problem_weak1}) on a triangulation $\Omega_h$, a matching mesh \cite[Definition 5.9]{Ric17} of the domain $\Omega$. If $N \in \mathbb{N}$ denotes the number of mesh nodes in $\Omega_h$, the total number of degrees of freedom $M \in \mathbb{N}$ that results after discretization is equal to $M=5N$ if $d=2$ and $M=7N$ if $d=3$ (one pressure and $d$ velocities and deformations, see \cite[Chapter 3]{WeiBR20}). Similar to the linear case, we first need the discrete differential operator (discretization matrix) $A_0 \in \mathbb{R}^{M \times M}$. $A_0$ discretizes all the linear operators involved in (\ref{equation_fsi_problem_weak1}) with a fixed shear modulus $\mu_s \in \mathbb{R}$ and a fixed kinematic fluid viscosity $\nu_f \in \mathbb{R}$. To be clear, restricted to the momentum equation,
\begin{align*}
A_0 \qquad \text{discretizes} \qquad \mu_s \langle \nabla u + \nabla u^T, \nabla \varphi \rangle_S+\lambda_s \langle \operatorname{tr}(\nabla u)I, \nabla \varphi  \rangle_S+\nu_f \rho_f \langle \nabla v+\nabla v^T, \nabla \varphi \rangle_F-\langle p, \nabla \cdot \varphi \rangle_F \text{.}
\end{align*}
Furthermore, we need the discrete operators $A_1$, $A_2 \in \mathbb{R}^{M \times M}$ such that
\begin{align*}
A_1 \qquad \text{discretizes} \qquad &\langle \nabla u + \nabla u^T, \nabla \varphi \rangle_S \qquad \text{and}  \qquad
A_2 \qquad \text{discretizes} \qquad  \langle \nabla v + \nabla v^T, \nabla \varphi \rangle_F \text{.}
\end{align*}
In our discrete finite element space, every unknown 
\begin{equation}\label{equalorder}
  x_h=\left( \begin{array}{c}p_h \\ v_h \\ u_h  
  \end{array} \right) \in X_h=V_h\times W_h\times L_h
\end{equation}
consists of a discrete pressure $p_h \in L_h$, a discrete velocity and deformation that are $v_h\in V_h$, $u_h \in W_h$.
Since all the operators discretized by the matrices $A_0$, $A_1$ and $A_2$ are linear, the corresponding operators can be discretized by choosing appropriate trial functions. The application  of linear operators, if the argument is contained in the finite element space of dimension $M$, is possible via matrix-vector products. For the nonlinear operators in (\ref{equation_fsi_problem_weak1})
\begin{align*}
\langle \nabla u^T \nabla u, \nabla \varphi \rangle_S
\text{,} \qquad \langle \operatorname{tr} (\nabla u^T \nabla u^T) I, \nabla \varphi \rangle_S \qquad \text{and} \qquad \langle (v \cdot \nabla) v, \varphi \rangle_F  \text{,}
\end{align*}this is not possible. 
\subsection{The Jacobian matrices}\label{chapter_jacobianmatrices1}
To linearize the nonlinear operators in (\ref{equation_fsi_problem_weak1}), we apply the Newton iteration \cite[Section 7.1.1]{book_quarteroni_newton1} due to its relevance and fast convergence rate when it comes to finite elements for the Navier-Stokes equations (compare \cite[Section 4.4]{Ric17} and \cite[Remark 6.44]{book_john_inc_flows1}). As an alternative, the Picard (fixed-point) iteration will be discussed in \cref{chap_picard_iteration1}. For linearization via the Newton iteration, we introduce the discrete Jacobian matrices for the nonlinear operators in (\ref{equation_fsi_problem_weak1}). $J_\rho(x_h) \in \mathbb{R}^{N \times N}$ is the discrete Jacobian matrix of
\begin{align*}
&J_{\langle (v \cdot \nabla)v, \varphi \rangle_F}(x_h)= \left( \begin{array}{ccc} 0 & 0 & 0\\
0& \frac{\partial \langle (v \cdot \nabla)v, \varphi \rangle_F}{\partial v}_{\big|_{v=v_h}} &0\\
0&0&0
\end{array}  \right) \in \mathbb{R}^{N \times N}
\end{align*}
$J_\mu(x_h) \in \mathbb{R}^{N \times N}$ the one of
\begin{align*}
J_{ \langle \nabla u+\nabla u^T , \nabla \varphi \rangle_S }(x_h)=\left( \begin{array}{ccc} 0& 0 &0\\
0& 0& \frac{\partial \langle \nabla u^T \nabla u, \nabla \varphi \rangle_S }{\partial u}_{\big|_{u=u_h}}
\\
0& 0& 0 
\end{array}  \right)
\in \mathbb{R}^{N \times N}
\end{align*}
and $J_\lambda(x_h) \in \mathbb{R}^{N \times N}$ the one of
\begin{align*}
J_{ \frac{1}{2}\langle \operatorname{tr}(\nabla u^T \nabla u )I, \nabla \varphi \rangle_S}(x_h)=\left(  \begin{array}{ccc}
0&0&0\\
0&0&\frac{\frac{1}{2}\partial \langle \operatorname{tr}(\nabla u^T \nabla u)I, \nabla \varphi \rangle_F }{\partial u}_{\big|_{u=u_h}} \\
0 &0&0
\end{array}    \right)\in \mathbb{R}^{N \times N} \text{.}
\end{align*}
\subsection{The Newton iteration}
At first, we fix $(i_1,i_2) \in \Xi^{m_1,m_2} := \{1,...,m_1\} \times \in \{1,...,m_2\}$ and take a look at the Newton iteration for the problem related to a shear modulus with value $\mu_s^{i_1}$ and a kinematic fluid viscosity of $\nu_f^{i_2}$. Even though the right hand side in (\ref{equation_fsi_problem_weak1}) vanishes, we incorporate the Dirichlet boundary conditions into the right hand side vector $b_D \in \mathbb{R}^M$. As initial guess, we choose $x^{i_1,i_2}_0:=b_D$. Let $g(x_{j}^{i_1,i_2},\mu_s^{i_1},\nu_f^{i_2})$ denote the residual of the problem (\ref{equation_fsi_problem_weak1}) with a shear modulus $\mu_s^{i_1}$ and a kinematic fluid viscosity $\nu_f^{i_2}$, evaluated at $x_j^{i_1,i_2} \in \mathbb{R}^M$ and let
\begin{equation}\label{equation_pardep_matrix1}
\begin{aligned}
A(x_{j-1}^{i_1,i_2},\mu_s^{i_1},\nu_f^{i_2}):=A_0&+ (\mu_s^{i_1}-\mu_s)A_1+(\nu_f^{i_2}-\nu_f)\rho_f A_2+ \mu_s^{i_1}J_\mu(x^{i_1,i_2}_{j-1})\\&+\lambda_s J_\lambda(x^{i_1,i_2}_{j-1})
+\rho_f J_\rho(x_{j-1}^{i_1,i_2}) \text{.}\end{aligned}
\end{equation}
At the Newton step $j \in \mathbb{N}$, we approximate the Newton update $\delta x_j^{i_1,i_2} \in X_h$ such that
\begin{equation}\label{equation_newton_step_indiv1}
\begin{aligned}
A(x_{j-1}^{i_1,i_2},\mu_s^{i_1}, \nu_f^{i_2})\delta x_j^{i_1,i_2} &=b_D-g(x_{j-1}^{i_1,i_2},\mu_s^{i_1},\nu_f^{i_2})\text{.}
\end{aligned}
\end{equation}
The Newton approximation of the Newton step $j$ is given as
\begin{align*}
x_j^{i_1,i_2}:=x_{j-1}^{i_1,i_2} +\delta x_j^{i_1,i_2}.
\end{align*}

Similar to the low-rank framework discussed in \cite{WeiBR20}, we are interested in translating a number of Newton steps of the form (\ref{equation_newton_step_indiv1}) to a single matrix equation. But since the matrix $A(x_{j-1}^{i_1,i_2},\mu_s^{i_1},\nu_f^{i_2})$ from (\ref{equation_pardep_matrix1}) depends, in addition to the parameters $\mu_s^{i_1}$ and $\nu_f^{i_2}$, on the Newton approximation of the previous linearization step, a direct translation like in \cite{WeiBR20} is not possible.

\subsection{Clustering of the parameter set}
\label{chapter_clustering1}

We first split the parameter set
\begin{align*}
S^{\mu, \nu}:=\{(\mu_s^{i_1},\nu_f^{i_2})\}_{\substack{i_1 \in \{1,...,m_1\}\\i_2 \in \{1,...,m_2\}}}  \subset \mathbb{R} \times \mathbb{R}
\end{align*}
into the disjoint subsets $\mathcal{I}_k$, namely
\begin{align*}
\mathcal{S}^{\mu, \nu}=\bigcupdot \limits_{k=1}^K \mathcal{I}_k \text{.}
\end{align*}
As mentioned in the introduction, the problems that are clustered within one subset $\mathcal{I}_k$ should not differ too much from each other. Grouping the right parameters to a cluster can be a challenge and depend on the parameter choice. The elements in the parameter set $S^{\mu,\nu}$ have to be ordered first. We define the most intuitive way to order the parameter set.
\begin{definition}[Little Endian Order]\label{definition_little_endian_order1}
  If for $(i_1, i_2)$, $(l_1,l_2) \in \Xi^{m_1,m_2}$, 
\begin{align*}
(\mu_1^{i_1},\eta_f^{i_2}) < (\mu_1^{l_1},\eta_f^{l_2}) \Leftrightarrow i_1+(i_2-1)m_1 <l_1+(l_2-1)m_1\text{,}
\end{align*}
the grouping operation of the indexes of the elements in $S^{\mu, \nu}$ is \emph{little endian}.


\end{definition}
There are other ways to order the elements in $S^{\mu, \nu}$. In a parameter-dependent discretization with respect to different shear moduli and first Lam\'{e} parameters, we might prefer to order the elements in the parameter set by its Poisson ratios. Even the sizes of the subsets $\{\mathcal{I}_k\}_{k \in \{1,...,K\}}$ may be chosen problem-dependently. In this paper, any clustering results in subsets $\{\mathcal{I}_k\}_{k \in \{1,...,K\}}$ with
\begin{align}\label{equation_clustering1}
|\mathcal{I}_k|=\begin{cases}  \lfloor \frac{m}{K} \rfloor &\text{if} \quad k \in \{1,...,K-1\}  \\ m-(K-1)\lfloor \frac{m}{K} \rfloor & \text{else.}
\end{cases}
\end{align}
\begin{example}\label{example_subsets_example1}
For $m_1=4$, $m_2=5$ and $K=3$, the subsets would be
\begin{equation}
\begin{aligned}
\mathcal{I}_1&=\{(\mu_s^1,\nu_f^1), (\mu_s^2,\nu_f^1),(\mu_s^3,\nu_f^1), (\mu_s^4,\nu_f^1), (\mu_s^1,\nu_f^2),(\mu_s^2,\nu_f^2)\} \text{,}\\
 \quad\mathcal{I}_2&=\{(\mu_s^3,\nu_f^2), (\mu_s^4,\nu_f^2),(\mu_s^1,\nu_f^3), (\mu_s^2,\nu_f^3), (\mu_s^3,\nu_f^3),(\mu_s^4,\nu_f^3)\} \qquad \text{and}\\
 \mathcal{I}_3&=\{(\mu_s^1,\nu_f^4), (\mu_s^2,\nu_f^4),(\mu_s^3,\nu_f^4), (\mu_s^4,\nu_f^4), (\mu_s^1,\nu_f^5),(\mu_s^2,\nu_f^5),(\mu_s^3,\nu_f^5), (\mu_s^4,\nu_f^5)\} \text{.}
\end{aligned}
\end{equation}
\end{example}
Details regarding clustering of the parameter set will be discussed in \cref{chap_alternative_clustering1}.
\begin{definition}[Upper Median Index]We call the index or multi-index that corresponds to the upper median of a set when ordered as defined in \cref{definition_little_endian_order1} the \emph{upper median index}.
\end{definition}
\begin{example}For the sets in \cref{example_subsets_example1}, the upper median index of $\mathcal{I}_1$ is $\tilde{m}^1=(4,1)$, the one of $\mathcal{I}_2$ is $\tilde{m}^2=(2,3)$ and the upper median index of $\mathcal{I}_3$ is $\tilde{m}^3=(1,5)$.
\end{example}

\subsection{A single Newton step formulated as matrix equation}\label{chapter_newton_step_formul1}

Let $\tilde{m}^k=(\tilde{m}^k_1,\tilde{m}^k_2) \in \mathbb{N} \times \mathbb{N}$ be the upper median index of the parameter set $\mathcal{I}_k$ for $k \in \{1,...,K\}$. Furthermore, let $x_{\epsilon_N}^{\tilde{m}^k}$ be the Newton approximation of the discretized FSI problem with a shear modulus $\mu_s^{\tilde{m}^k_1}$ and a kinematic fluid viscosity  $\nu_f^{\tilde{m}^k_2}$. $\epsilon_N>0$ defines the accuracy of the Newton approximation $x_{\epsilon_N}^{\tilde{m}^k}$. To be more precise, the stopping  criterion for the Newton iteration is fulfilled whenever the residual norm is smaller than $\epsilon_N$.
\begin{definition}[Operator $\diago$]We extend the definition \cite[Section 1.2.6]{book_golub_matrixcomp1} to
\begin{align}\label{equation_diagoblockdef1}
\diago \limits_{i \in \{1,...,m\}} (A_i)=\left(\begin{array}{ccc} A_1& &  \substack{\mbox{}\\ \text{\LARGE$0$} }\\\
&\ddots&\\
\text{\LARGE$0$} &&A_m
\end{array}\right)\in \mathbb{R}^{Mm\times Mm} 
\end{align} 
if the argument of the operator $\diago(\cdot)$ is a set of matrices
\begin{align*}
A_i \in \mathbb{R}^{M \times M} \qquad \text{for all} \qquad i \in \{1,...,m\}\text{.}
\end{align*}
(\ref{equation_diagoblockdef1}) is block diagonal. If the argument of the operator $\diago(\cdot)$ is a single square matrix, we have
\begin{align*}
\diago(A)=\left( \begin{array}{c}a_{11}\\ \vdots \\ a_{MM}
\end{array} \right)\in \mathbb{R}^M \qquad \text{for} \qquad A=\left(\begin{array}{ccc} a_{11} & \cdots &a_{1M}\\
\vdots & \ddots & \vdots\\
a_{M1}& \cdots & a_{MM}
\end{array} \right) \in \mathbb{R}^{M \times M}\text{.}
\end{align*}
\end{definition}
\begin{definition}[Diagonal Matrices]\label{remark_diag_matrices1}
Similar to the diagonal matrices from (\ref{equation_diag_intro1}),
\begin{align*}
D_{\mu}^k=\diago \limits_{(\mu_s^{i_1},\nu_f^{i_2})\in \mathcal{I}_k}(\mu_s^{i_1})\qquad \text{and} \qquad D_{\nu}^k=\diago \limits_{(\mu_s^{i_1},\nu_f^{i_2})\in \mathcal{I}_k}(\nu_f^{i_2})\in \mathbb{R}^{|\mathcal{I}_k| \times |\mathcal{I}_k|}\text{,}
\end{align*}
 for $k \in \{1,...,K\}$, we define
\begin{align*}
D_{\mu-}^k:=D_\mu^k-\mu_s I^{|\mathcal{I}_k| \times |\mathcal{I}_k|} \qquad \text{and} \qquad D_{\nu-}^k:=D_\nu^k-\nu_fI^{|\mathcal{I}_k| \times |\mathcal{I}_k|} \in \mathbb{R}^{|\mathcal{I}_k| \times |\mathcal{I}_k|}\text{,}
\end{align*}
where $I^{|\mathcal{I}_k| \times |\mathcal{I}_k|}$ denotes the $|\mathcal{I}_k| \times |\mathcal{I}_k|$ identity matrix.
\end{definition} 
With this notation, we formulate a Newton step on the whole subset $\mathcal{I}_k$ that uses $x_{\epsilon_N}^{\tilde{m}^k}$ as initial guess.

\label{chapter_newton_step_on_Ik}
We first formulate the Newton step on $\mathcal{I}_k$ in the vector notation, similar to \cite[Definition 2]{WeiBR20}. To do so, we need the vectorization operator.
\begin{definition}[Vectorization restricted to $\mathbb{R}^{M \times m}$]
For a matrix
\begin{align*}
(v_1|...|v_m ) \in \mathbb{R}^{M \times m} \quad \text{with columns} \quad v_i \in \mathbb{R}^M \quad \text{for} \quad i \in \{1,...,m\}\text{,}
\end{align*}
the vectorization operator \cite[Section 5.1]{hackbusch_tensor1} is defined as
\begin{align*}
\operatorname{vec}: \mathbb{R}^{M \times m} \rightarrow \mathbb{R}^{Mm} \text{, } \operatorname{vec}(v_1 | ...| v_m )\mapsto \left( \begin{array}{c} v_1 \\ \vdots \\ v_m \end{array} \right) \in \mathbb{R}^{Mm}\text{.}
\end{align*}
The inverse vectorization operator is defined as
\begin{align*}
\operatorname{vec}^{-1}: \mathbb{R}^{Mm} \rightarrow \mathbb{R}^{M \times m}\text{,} \operatorname{vec}^{-1}\left(\begin{array}{c}v_1\\ \vdots\\v_m
\end{array} \right)\mapsto (v_1|...|v_m) \in \mathbb{R}^{M \times m}\text{.}
\end{align*}
\end{definition}

Once the initial guess for one Newton step is fixed to $x_{\epsilon_N}^{\tilde{m}^k}$ for $k \in \{1,...,K\}$, a Newton step for all problems related to the parameters in $\mathcal{I}_k$ can be formulated as follows: Find $s_k \in \mathbb{R}^{M|\mathcal{I}_k|}$ such that
\begin{align}\label{equation_newtonstep_vectornot1}
\underbrace{\diago \limits_{(\mu_s^{i_1},\nu_f^{i_2})\in \mathcal{I}_k} A(x_{\epsilon_N}^{\tilde{m}^k},\mu_s^{i_1},\nu_f^{i_2})}_{=:\mathcal{A}_k} s_k=b_D\otimes (1,...,1)^T-\Big(g(x_{\epsilon_N}^{\tilde{m}^k},\mu_s^{i_1},\nu_f^{i_2})\Big)_{(\mu_s^{i_1},\nu_f^{i_2})\in \mathcal{I}_k} \text{.}
\end{align}
$\mathcal{A}_k \in \mathbb{R}^{M|\mathcal{I}_k| \times M|\mathcal{I}_k|}$ is a block diagonal matrix. The unknown, the Newton update $s_k$, is a vector. Therefore, we will call the notation used in (\ref{equation_newtonstep_vectornot1}) the \emph{vector notation}. In this notation, the approximation for the next Newton step is
\begin{align*}
x^k:=\left( \begin{array}{c}x_{\epsilon_N}^{\tilde{m}^k}\\ \vdots \\ x_{\epsilon_N}^{\tilde{m}^k}
\end{array} \right)+s_k \in \mathbb{R}^{M|\mathcal{I}_k|} \text{.}
\end{align*}

We now translate (\ref{equation_newtonstep_vectornot1}) to a matrix equation.
Let
\begin{equation}\label{equation_rhs_matrix1}
\begin{aligned}
B_k^{\mu,\nu}(x_{\epsilon_N}^{\tilde{m}^k})&:=\big(b_D-g(x_{\epsilon_N}^{\tilde{m}^k},0,0)\big) \otimes (1,...,1)- \big( A_1 x_{\epsilon_N}^{\tilde{m}^k}+g_\mu(x_{\epsilon_N}^{\tilde{m}^k})\big) \otimes \diago(D_\mu^k)^T\\
&\hspace{.39cm}-\rho_fA_2x_{\epsilon_N}^{\tilde{m}^k}\otimes \diago(D_\nu^k)^T\text{,}
\end{aligned}
\end{equation}
where $g_\mu(x_h)$ evaluates the operator
\begin{align*}
\langle \nabla u^T \nabla u, \nabla \varphi \rangle_S
\end{align*}
at the given argument $u=u_h$. On the subset $\mathcal{I}_k$, for $k \in \{1,...,K\}$, the Newton step (\ref{equation_newtonstep_vectornot1}) is equivalent to the following matrix equation:
Find $S_k \in \mathbb{R}^{M\times |\mathcal{I}_k|}$ such that
\begin{equation}\label{equation_newton_matrix1}
\begin{aligned}
\underbrace{A_0S_k+A_1S_kD_{\mu-}^k+\rho_fA_2S_kD_{\nu-}^k+J_\mu (x_{\epsilon_N}^{\tilde{m}^k})S_k D_{\mu}^k+\lambda_sJ_\lambda(x_{\epsilon_N}^{\tilde{m}^k})S_k
+\rho_f J_\rho(x_{\epsilon_N}^{\tilde{m}^k})S_k}_{=:F(S_k, x_{\epsilon_N}^{\tilde{m}^k})}&=B_k^{\mu,\nu}(x_{\epsilon_N}^{\tilde{m}^k}) \text{.}
\end{aligned}
\end{equation}
The approximation for the next Newton step is given by
\begin{align}\label{equation_matrix_nextiterate1}
X^k:=x_{\epsilon_N}^{\tilde{m}^k}\otimes (1,...,1)+S_k \in \mathbb{R}^{M\times |\mathcal{I}_k|}\text{.}
\end{align}
\begin{remark}
Note that 
\begin{align*}
\operatorname{vec}(S_k)=s_k \qquad \text{as well as} \qquad \operatorname{vec}(X^k)=x^k\text{.}
\end{align*}
\end{remark}

The global approximation to the parameter-dependent problem is
\begin{align*}
\tilde{X}=[X^1|\cdots|X^K]\text{.}
\end{align*}
If the order of the parameters is chosen as defined in \cref{definition_little_endian_order1}, the column $p \in \mathbb{N}$ of $\tilde{X}$ is an approximation of the finite element solution to the FSI problem related to the parameters
\begin{align*}
(\mu_s^{\scalebox{.93}{(}(p-1)\hspace{-.18cm}\mod m_1\scalebox{.93}{)}+1}\text{, }\nu_f^{ \lceil \frac{p}{m_1} \rceil}) \text{.}
\end{align*}
We now explain how to approximate the matrix $S_k$ in (\ref{equation_newton_matrix1}) and why we do not apply more than one Newton step on $\mathcal{I}_k$.

\section{Low-rank methods and generalization}\label{chap_lrank_meth_and_gen1}
We first examine, from a theoretical point of view, the low-rank approximability of $S_k$ in (\ref{equation_newton_matrix1}).
Consider the parameter-dependent matrix $A(x_{\epsilon_N}^{\tilde{m}^k},\mu_s^{i_1},\nu_f^{i_2})$ and the right hand side
\begin{align*}
b(\mu_s^{i_1},\nu_f^{i_2}):=b_D-g(x_{\epsilon_N}^{\tilde{m}^k},\mu_s^{i_1},\nu_f^{i_2})
\end{align*}
of the equations involved in the block diagonal problem (\ref{equation_newtonstep_vectornot1}). Once $x_{\epsilon_N}^{\tilde{m}^k}$ is computed for $k \in \{1,...,K\}$, we consider $A(x_{\epsilon_N}^{\tilde{m}^k},\mu_s^{i_1},\nu_f^{i_2})$ as a matrix-valued function $A(\mu_s^{i_1},\nu_s^{i_2})$ that depends on the parameters $\mu_s^{i_1}$, $\nu_f^{i_2}$ only. 

\subsection{Approximability of \texorpdfstring{$\boldsymbol{S_k}$}{Sk} by a low-rank matrix}\label{section_approximability1}

Without loss of generality, we transform, for \cref{section_approximability1}, the parameter set $S^{\mu,\nu}$ such that
\begin{align*}
S^{\mu,\nu} \subset [-1,1] \times [-1,1]\text{.}
\end{align*}
Let $\mathcal{E}_{\rho_0}\subset \mathbb{C}$ be the open elliptic disc with foci $\pm 1$ and sum of half axes of $\rho_0>0$. If, in a one-parameter discretization, $A(\cdot)$ and $b(\cdot)$ are assumed to have analytic extensions on $\mathcal{E}_{\rho_0}$ and $A(\cdot)$ is assumed to be invertible on $\mathcal{E}_{\rho_0}$, the singular value decay of the matrix $S_k$ in (\ref{equation_newton_matrix1}) is exponential  as proven in \cite[Theorem 2.4]{KreT11}.
%
We define the open elliptic polydisc
\begin{align*}
  \mathcal{E}_{\rho_0}^\times:=\mathcal{E}_{\rho_0}\times \mathcal{E}_{\rho_0}\text{.}
\end{align*}

\begin{corollary}[Case $p=2$ of Theorem 3.6 in \cite{KreT11}]
  \label{corollary_theorem3_6}
  Assume that  
  \begin{align*}
    b: [-1,1]\times [-1,1] \rightarrow \mathbb{R}^M \qquad \text{and} \qquad A: [-1,1]\times [-1,1] \rightarrow \mathbb{R}^{M \times M}
  \end{align*}have analytic extensions on $\mathcal{E}_{\rho_0}^\times$ and $A(\mu, \nu) $ is invertible for all $(\mu,\nu)\in \mathcal{E}_{\rho_0}^\times$. There exists $\hat{S}_k\in \mathbb{R}^{M \times M}$ of rank $R \in \mathbb{N}$ for any $t=\frac{1}{q}-1$ with $0 < q \leq 1$ such that
  \begin{align*}
    \|S_k-\hat{S}_k\|_F \leq \sqrt{Mm_1m_2} C R^{-t}\text{,}
  \end{align*}
  with $C$ as defined in \cite[Theorem 3.6]{KreT11}.
\end{corollary}
\begin{proof}
  \cref{corollary_theorem3_6} follows directly from \cite[Theorem
    3.6]{KreT11}.
\end{proof}

\begin{algorithm}[t]
  \caption{Two-parameter Discretization of Nonlinear FSI Problem} \label{algorithm_one1}
  \begin{algorithmic}
    \Require{Accuracy $\epsilon_N>0$, ranks $R_k \in \mathbb{N}$ for $k \in \{1,...,K\}$}
    \Ensure{$\hat{X}$, a rank-$\sum \limits_{k=1}^K R_k$ approximation of the parameter-dependent FSI discretization}
    \State Split the parameter set $S^{\mu, \nu}$ into the disjoint subsets $\bigcupdot \limits_{k=1}^K \mathcal{I}_k$
    \For{$k = 1,...,K$} 
    \State Compute the Newton approximation $x_{\epsilon_N}^{\tilde{m}^k}$ of the problem related to the upper median parameters\\
    \hspace{.55cm}in $\mathcal{I}_k$ with index $\tilde{m}^k$. Stopping criterion $\|residual\|_2 \leq \epsilon_N$
    \State Use $x_{\epsilon_N}^{\tilde{m}^k}$ as initial guess for one Newton step on the subset $\mathcal{I}_k$. By means of a low-rank method \\ \hspace{.55cm}such as GMREST or the ChebyshevT method from \cite{WeiBR20}, find a rank-$R_k$ approximation \begin{align*}
      \hat{X}^k \approx x_{\epsilon_N}^{\tilde{m}^k}\otimes (1,...,1)+S_k \in \mathbb{R}^{M \times |\mathcal{I}_k|}\text{,}
    \end{align*}
    \\ \hspace{.55cm}where $S_k$ is a solution to
    \begin{align}\label{equation_matrixeq_algo1}
      F(S_k, x_{\epsilon_N}^{\tilde{m}^k})=B_k^{\mu,\nu}(x_{\epsilon_N}^{\tilde{m}^k})\text{,}
    \end{align}
    \\ \hspace{.55cm} with the notation from (\ref{equation_newton_matrix1}).
    \EndFor
    \State The global approximation is then given by
    \begin{align*}
      \hat{X}:=[\hat{X}^1|\cdots|\hat{X}^K]\text{.}
\end{align*}
  \end{algorithmic}
\end{algorithm}

\begin{remark}
  Even though the term $R^{-t}$ suggests a polynomial decay of the error for an increasing rank $R$, one has to be careful about the constant $C$. The smaller $q$ is chosen, the bigger $C$ becomes. Choosing the right $q$ is difficult. However, we do not go into detail here and refer to \cite[Section 3.1]{KreT11} for further reading.
\end{remark}

The matrix $S_k$ in (\ref{equation_newton_matrix1}) can now be approximated by a low-rank method such as the GMREST or the ChebyshevT method from \cite[Algorithm 1 and 3]{WeiBR20}. Once $\hat{S}_k$, a low-rank approximation of $S_k$, is computed, a low-rank approximation of $X^k$ in (\ref{equation_matrix_nextiterate1}) can be achieved by a simple rank $1$ update.

\subsection{Preconditioner}\label{chapter_preconditioner1}

The mean-based preconditioner corresponding to the formulation (\ref{equation_newtonstep_vectornot1}) in vector notation is
\begin{align*}
\mathcal{P}_T^k&:=I^{|\mathcal{I}_k|\times |\mathcal{I}_k|}\otimes \underbrace{A(x_{\epsilon_N}^{\tilde{m}^k}\text{, }\bar{\mu}_s^k\text{, }\bar{\nu}_f^k)}_{=:P_T^k}\text{,}
\end{align*}
where
\begin{align*}
\bar{\mu}_s^k&:=\frac{\min \limits_{(\mu_s^{i_1},\nu_f^{i_2})\in \mathcal{I}_k} (\mu_s^{i_1}-\mu_s) + \max \limits_{(\mu_s^{i_1}, \nu_f^{i_2}) \in \mathcal{I}_k} (\mu_s^{i_1}-\mu_s)}{2}\qquad \text{and}\\
\bar{\nu}_f^k&:=\frac{\min  \limits_{(\mu_s^{i_1}, \nu_f^{i_2}) \in \mathcal{I}_k} (\nu_f^{i_2}-\nu_f)+\max \limits_{(\mu_s^{i_1}, \nu_f^{i_2}) \in \mathcal{I}_k} (\nu_f^{i_2}-\nu_f) }{2}
\qquad \text{for} \qquad k \in \{1,...,K\}\text{,}
\end{align*}
similar to $\mathcal{P}_T$ in \cite[Section 3.2]{WeiBR20}.
\begin{remark}[Multiple Newton Steps]We point out two difficulties that would come up if multiple Newton steps in the form (\ref{equation_newton_matrix1}) were performed. In a second Newton step, $\hat{X}^k$, a low-rank approximation of $X^k$ in (\ref{equation_matrix_nextiterate1}), serves as initial guess.
\par In the first Newton step, the right hand side matrix $B_k^{\mu,\nu}(x_{\epsilon_N}^{\tilde{m}^k})$ in (\ref{equation_newton_matrix1}) has rank not more than $3$. Even if the initial guess $\hat{X}^k$ for a second Newton step is a low-rank matrix, the columns of $\hat{X}^k$ differ from each other. Therefore, the residual of the problem, the right hand side, to be more precise, $g(\cdot,\cdot,\cdot)$ has to be evaluated column by column $|\mathcal{I}_k|$ times. This is too expensive and the low-rank structure of the right hand side matrix is not assured anymore in a second Newton step.
\par The first Newton step is indeed a Newton step. Due to the fact that all problems use the same initial guess $x_{\epsilon_N}^{\tilde{m}^k}$, the Jacobian matrix assembled is correct for all problems related to the subset $\mathcal{I}_k$. In a second Newton step, this would not be the case. Even for different initial guesses, one single Jacobian matrix would be used for all problems related to the respective subset $\mathcal{I}_k$.
\end{remark}
The resulting method, a low-rank method for a two-parameter discretization of (\ref{equation_fsi_problem_weak1}), is coded in \cref{algorithm_one1}.

\subsection{Time dependent fluid-structure interactions}\label{chapter_time1}

Similar to \cite[Section 5.1]{WeiBR20}, we consider the time-dependent nonlinear fluid-structure interaction problem that couples the Navier-Stokes with the St. Venant Kirchhoff model equations. 
Let $t \in [0,T] \subset \mathbb{R}$ be the time variable for $T \in \mathbb{R}$ and $\rho_s \in \mathbb{R}$ denote the solid density. The weak formulation of the non-stationary nonlinear fluid-structure interaction problem is
\begin{equation}
\label{equation_fsi_problem_time_weak1}
\begin{aligned}
\langle \nabla \cdot v, \xi  \rangle_F&=0\text{,}\\
\mu_s \langle \nabla u + \nabla u^T+\nabla u^T \nabla u, \nabla \varphi \rangle_S +\lambda_s \langle \operatorname{tr}(\nabla u+\frac{1}{2}\nabla u^T \nabla u)I, \nabla \varphi \rangle_S+\underbrace{\rho_s \langle \partial_t v, \varphi \rangle_S}_{(\star)}\\
+\underbrace{\rho_f \langle \partial_t v, \varphi \rangle_F}_{(\star \star)} +\rho_f \langle (v \cdot \nabla)v, \varphi  \rangle_F + \nu_f \rho_f \langle \nabla v+\nabla v^T, \nabla \varphi \rangle_F- \langle p, \nabla \cdot \varphi \rangle_F &=0 \quad \text{and}\\
\langle \nabla u, \nabla \psi  \rangle_F&=0\text{,}
\end{aligned}
\end{equation}
with $v \in L^2\big([0,T];v_{\text{in}}+H_0^1(\Omega,\Gamma_f^D \cup \Gamma_{\text{int}})^d\big)$, $\partial_t v \in L^2\big([0,T];H^{-1}(\Omega)^d\big)$ for all $(t,x) \in [0,T] \times \Omega$. We use the notation from (\ref{equation_fsi_problem_weak1}). For time discretization, we apply the $\theta$-scheme from \cite[Section 4.1]{Ric17}.
\subsubsection{Time discretization with the \texorpdfstring{$\boldsymbol{\theta}$}{Theta}-scheme}
Let the discretization matrices $A_t^f$, $A_t^s \in \mathbb{R}^{M \times M}$ be such that
\begin{align*}
A_t^f \quad \text{discretizes} \quad \langle v,\varphi \rangle_F \quad \text{and} \quad A_t^s \quad \text{discretizes} \quad \langle v, \varphi\rangle_S \text{.}
\end{align*}
We consider the discrete time interval $[0,T]$ that is split into $w+1 \in \mathbb{N}$ equidistant time steps. The start time is $t_0=0$, the time steps are given by $t_i:=i\Delta_t$ for $i \in \{0,...,w\}$, $\Delta_t=t_{i+1}-t_{i}$ for all $i \in \{0,...,w-1\}$ and $w\Delta_t=T$. Let $b_D^i \in \mathbb{R}^M$ be the right hand side at time $t_i$ and
\begin{align*}
B^i:=b_D^i \otimes (1,...,1) \quad \text{for} \quad i \in \{0,...,w\}\text{.}
\end{align*}
Since both time-dependent operators in (\ref{equation_fsi_problem_time_weak1}), $(\star)$ on the solid and $(\star \star)$ on the fluid, depend linearly on the unknown, a parameter-dependent discretization is straight forward.

We start to discuss the problem with fixed parameters $(\mu_s^{i_1},\nu_f^{i_2})\in S^{\mu,\nu}$. Let $x^{i_1,i_2,t_0}$ at time $t_0=0$ be the initial. We can compute $x^{i_1,i_2,t_0}$ as an approximation to the stationary problem or simply set $x^{i_1,i_2,t_0}=b_D^0$. Let $\theta \in (0,1]$. $x_{\epsilon_N}^{i_1,i_2,t_{i-1}} \in \mathbb{R}^M$ denotes the approximation of the previous time step.
At the Newton step $j \in \mathbb{N}$, find $s\in \mathbb{R}^M$ such that
\begin{multline}\label{equation_time_indiv1}
  \Big(\frac{1}{\Delta_t}(\rho_f A_t^f+\rho_s A_t^s)+ \theta A(x_{j-1}^{i_1,i_2,t_i},\mu_s^{i_1}, \nu_f^{i_2})\Big)\delta x_{j}^{i_1,i_2,t_i}=\theta b_D^{i}+(1-\theta)b_D^{i-1}\\
  +\frac{1}{\Delta_t}(\rho_f A_t^f+\rho_s A_t^s)(x_{\epsilon_N}^{i_1,i_2,t_{i-1}}-x_{j-1}^{i_1,i_2,t_i})
  -(1-\theta)g(x_{\epsilon_N}^{i_1,i_2,t_{i-1}},\mu_s^{i_1},\nu_f^{i_2})-\theta g(x_{j-1}^{i_1,i_2,t_i},\mu_s^{i_1},\nu_f^{i_2})\text{,}
\end{multline}
with the notation from (\ref{equation_newton_step_indiv1}). The approximation at the next Newton step is given by
\begin{align*}
  x_{j}^{i_1,i_2,t_i}:=x_{j-1}^{i_1,i_2,t_i}+\delta x_j^{i_1,i_2,t_i}.
\end{align*}
After $l \in \mathbb{N}$ Newton steps, the approximation of problem (\ref{equation_fsi_problem_time_weak1}) at time $t_i$, related to the parameter combination $(\mu_s^{i_1},\nu_f^{i_2})$, is given by $x_l^{i_1,i_2,t_i} \in \mathbb{R}^{M}$. (\ref{equation_time_indiv1}) can be, for a set of problems, translated to one single matrix equation similar to (\ref{equation_newton_matrix1}). Now, the Newton update at time step $t_i$, $S_k^{t_i} \in \mathbb{R}^{M \times |\mathcal{I}_k|}$, is time-dependent.

We consider the problems related to the parameters in the subset $\mathcal{I}_k$. Let $X^{k,t_{i}} \in \mathbb{R}^{M \times |\mathcal{I}_k|}$ be the approximation at time step $t_{i}$ and $x_{\epsilon_N}^{\tilde{m}^k,t_i} \in \mathbb{R}^M$ the Newton approximation of the upper median parameter problem for $i \in \{1,...,w\}$. $x_{\epsilon_N}^{\tilde{m}^k,t_i}$ is computed via (\ref{equation_time_indiv1}) for $(i_1,i_2)=\tilde{m}^k$. \par At time $t_0$, $X^{k,t_0} \in \mathbb{R}^{M \times |\mathcal{I}_k|}$ is given as initial value as well as $x_{\epsilon_N}^{\tilde{m}^k,t_0} \in \mathbb{R}^M$. On the subset $\mathcal{I}_k$ at time step $t_i$ for $i>0$, we apply the $\theta$-scheme for $\theta\in (0,1]$. In the Newton step on $\mathcal{I}_k$, we find $S_k^{t_i}\in \mathbb{R}^{M \times |\mathcal{I}_k|}$ such that
\begin{multline}\label{equation_timedep_matrix1}
  \frac{1}{\Delta_t}\big(\rho_f  A_t^f  + \rho_s A_t^s\big)S_k^{t_i}+\theta F(S_k^{t_i},x_{\epsilon_N}^{\tilde{m}^k,t_i})=\theta B^i+(1-\theta )B^{i-1}\\
  +\frac{1}{\Delta_t}(\rho_f A_t^f+\rho_s A_t^s)\big(X^{k,t_{i-1}}-x_{\epsilon_N}^{\tilde{m}^k,t_i}\otimes (1,...,1)\big)\\
  -\underbrace{(1-\theta)\operatorname{vec}^{-1}\Big( \big( g(X_{\mu_s^{i_1},\nu_f^{i_2}}^{k,t_{i-1}},\mu_s^{i_1},\nu_f^{i_2}  )\big)_{(\mu_s^{i_1},\nu_f^{i_2})\in \mathcal{I}_k}\Big)}_{(\star \star \star)}
  -\theta B_k^{\mu, \nu}(x_{\epsilon_N}^{\tilde{m}^k,t_i})\text{,}
\end{multline}
where $X_{\mu_s^{i_1},\nu_f^{i_2}}^{k.t_{i-1}}$ denotes the approximation of the last time step related to the parameters $(\mu_s^{i_1},\nu_f^{i_2})$.
\begin{remark}[Right Hand Side]Evaluation of the right hand side in  (\ref{equation_timedep_matrix1}), especially $(\star \star \star)$, is expensive. $(\star \star \star)$ could be approximated by
\begin{align*}
(1-\theta) B_k^{\mu, \nu}(x_{\epsilon_N}^{\tilde{m}^k ,t_{i-1}})\text{,}
\end{align*}
which has  at most rank $3$ and, therefore, is cheaper to evaluate.
\end{remark}
\begin{remark}[Implicit Treatment of the Pressure]As discussed in \cite[Section 4.1.4]{Ric17}, the pressure operator
\begin{align*}
-\langle p,\nabla \cdot \varphi \rangle_F
\end{align*}
is always treated implicit.
\end{remark}
Equation (\ref{equation_timedep_matrix1}) is, in matrix notation, a Newton step of the discretization of (\ref{equation_fsi_problem_time_weak1}) on $\mathcal{I}_k$ at time step $t_i$ with initial guess $x_{\epsilon_N}^{\tilde{m}^k,t_i}$. Hence, parameter-dependent discretizations of (\ref{equation_fsi_problem_time_weak1}) can be approached by low-rank methods that approximate the Newton update $S_k^{t_i}$ at time step $t_i$, the solution to (\ref{equation_timedep_matrix1}).

\subsection{Extension to further parameters}
\label{chap_extension_further1}

Consider a parameter-dependent discretization of problem (\ref{equation_fsi_problem_weak1}) with respect to $m_1$ shear moduli, $m_2$ kinematic fluid viscosities, $m_3$ first Lam\'{e} parameters and $m_4$ fluid densities. The parameter set is now given by
\[
  S^{\mu, \nu, \lambda,\rho}:=\{(\mu_s^{i_1},\nu_f^{i_2},\lambda_s^{i_3},\rho_f^{i_4})\}_{i_k\in \{1,\dots,m_k\},\; k=1,2,3,4}\subset \mathbb{R} \times \mathbb{R} \times \mathbb{R} \times \mathbb{R}.
\]
We fix the parameters $(\mu_s^{i_1},\nu_f^{i_2},\lambda_s^{i_3},\rho_f^{i_4}) \in S^{\mu,\nu,\lambda,\rho}$ and consider the Newton iteration for the problem related to this parameter choice. Let the discrete operator $A_3 \in \mathbb{R}^{M \times M}$ be such that
\[
A_3 \qquad \text{discretizes} \qquad \langle \operatorname{tr}(\nabla u)I,\nabla \varphi \rangle_S\text{.}
\]
With $x_0^{i_1,i_2,i_3,i_4}:=b_D$, the parameter-dependent Jacobian matrix of the Newton step $j \in \mathbb{N}$ is
\begin{align*}
A(x_{j-1}^{i_1,i_2,i_3,i_4},\mu_s^{i_1},\nu_f^{i_2},\lambda_s^{i_3},\rho_f^{i_4})&:=A_0+(\mu_s^{i_1}-\mu_s)A_1+(\nu_f^{i_2}\rho_f^{i_4}-\nu_f \rho_f)A_2+(\lambda_s^{i_3}-\lambda_s)A_3
\\
&\hspace{.3cm}+\mu_s^{i_1}J_\mu(x_{j-1}^{i_1,i_2,i_3,i_4})+\lambda_s^{i_3}J_\lambda(x_{j-1}^{i_1,i_2,i_3,i_4})+\rho_f^{i_4}J_\rho(x_{j-1}^{i_1,i_2,i_3,i_4})\text{.}
\end{align*}
Find the Newton update $\delta x_j^{i_1,i_2,i_3,i_4}\in X_h$ such that
\begin{align*}
A(x_{j-1}^{i_1,i_2,i_3,i_4},\mu_s^{i_1},\nu_f^{i_2},\lambda_s^{i_3},\rho_f^{i_4})\delta x_j^{i_1,i_2,i_3,i_4}=b_D-g(x_{j-1}^{i_1,i_2,i_3,i_4},\mu_s^{i_1},\nu_f^{i_2},\lambda_s^{i_3},\rho_f^{i_4}),
\end{align*}
where $g(x_{j-1}^{i_1,i_2,i_3,i_4},\mu_s^{i_1},\nu_f^{i_2},\lambda_s^{i_3},\rho_f^{i_4})$ denotes the residual of the problem (\ref{equation_fsi_problem_weak1}) evaluated at $x_{j-1}^{i_1,i_2,i_3,i_4}$ and parameters $(\mu_s^{i_1}, \nu_f^{i_2},\lambda_s^{i_3},\rho_f^{i_4})$. The approximation for the next linearization step is given by
\begin{align*}
x_j^{i_1,i_2,i_3,i_4}=x_{j-1}^{i_1,i_2,i_3,i_4}+\delta x_j^{i_1,i_2,i_3,i_4}.
\end{align*}

\subsubsection{Clustering}
Again, the grouping operation of the indexes of the elements in $S^{\mu,\nu,\lambda,\rho}$ is little endian. For
\begin{align*}
(i_1,i_2,i_3,i_4)\text{, }(l_1,l_2,l_3,l_4) \in \{1,...,m_1\}\times \{1,...,m_2\}\times \{1,...,m_3\}\times \{1,...,m_4\}\text{,}
\end{align*}
it holds
\begin{align*}
(\mu_s^{i_1},\nu_f^{i_2},\lambda_s^{i_3},\rho_f^{i_4}) &< (\mu_s^{l_1},\nu_f^{l_2},\lambda_s^{l_3},\rho_f^{l_4}) \\
&\Updownarrow\\
i_1+(i_2-1)m_1+(i_3-1)m_1m_2 
+(i_4-1)m_1m_2m_3 &<l_1+(l_2-1)m_1+(l_3-1)m_1m_2+(l_4-1)m_1m_2m_3 \text{.}
\end{align*}
Let the clustering of the parameter set 
\begin{align*}
S^{\mu,\nu,\lambda,\rho}=\bigcupdot \limits_{k=1}^K \mathcal{I}_k
\end{align*} be as defined in (\ref{equation_clustering1}). For the sake of readability, we define
\begin{align*}
\Xi^k :=\{(i_1,i_2,i_3,i_4) : (\mu_s^{i_1},\nu_f^{i_2},\lambda_s^{i_3},\rho_f^{i_4}) \in \mathcal{I}_k\}
\end{align*}and simply write
\begin{align*}
(i_2,i_4)\in \Xi^k \quad \text{for the set} \quad \{ (i_2,i_4) : (i_1,i_2,i_3,i_4) \in \Xi^k\}\text{.}
\end{align*}
\subsubsection{Diagonal matrices}
To formulate the Newton step on $\mathcal{I}_k$ as a matrix equation, we need the matrices
\begin{align*}
D_{\nu\rho}^k&:=\diago \limits_{(i_2,i_4) \in \Xi^k} (\nu_f^{i_2}\rho_f^{i_4})\text{,} \qquad  D_{\nu\rho-}^k:=D_{\nu\rho}^k-\nu_f\rho_fI^{|\mathcal{I}_k|\times |\mathcal{I}_k| }\text{,} \qquad
&&D_{\lambda}^k:=\diago \limits_{
i_3\in \Xi^k} (\lambda_s^{i_3}) \text{,} \\
D_{\lambda-}^k&:=D_\lambda^k-\lambda_sI^{|\mathcal{I}_k|\times|\mathcal{I}_k|}\text{,}\qquad \hspace{.355cm} D_\rho^k:=\diago \limits_{i_4 \in \Xi^k} (\rho_f^{i_4}) \qquad \qquad \text{and} \qquad &&\hspace*{-.13cm}D_{\rho-}^k:=D_\rho^k-\rho_fI^{|\mathcal{I}_k|\times|\mathcal{I}_k|}
\end{align*}
in addition to the diagonal matrices from \cref{remark_diag_matrices1}.
\subsubsection{The matrix equation for four parameters} Let $\tilde{m}^k=(\tilde{m}_1^k,\tilde{m}_2^k,\tilde{m}_3^k,\tilde{m}_4^k)\in \Xi^k$ be the index corresponding to the upper median parameter of the subset $\mathcal{I}_k$ and $x_{\epsilon_N}^{\tilde{m}^k} \in \mathbb{R}^M$ the Newton approximation of the related problem. In addition to the evaluation function $g_\mu(\cdot)$ from (\ref{equation_rhs_matrix1}), we define $g_\lambda(x_h)$ and $g_\rho(x_h)$ such that they evaluate the operators
\begin{align*}
\frac{1}{2}\langle \operatorname{tr}(\nabla u^T \nabla u)I,\nabla \varphi \rangle_S \qquad \text{and} \qquad \langle (v \cdot \nabla)v,\varphi \rangle_F
\end{align*}
at the given argument $u=u_h$ and $v=v_h$, respectively. The right hand side in the matrix equation is
\begin{align*}
B_k^{\mu,\nu,\lambda,\rho}(x_{\epsilon_N}^{\tilde{m}^k})&:= \big(b_D -g(x_{\epsilon_N}^{\tilde{m}^k},0,0,0,0)\big)\otimes (1,...,1)-\big(A_1 x_{\epsilon_N}^{\tilde{m}^k}+g_\mu(x_{\epsilon_N}^{\tilde{m}^k})\big)\otimes \diago(D_\mu^k)^T \\
&\hspace{.39cm} -A_2x_{\epsilon_N}^{\tilde{m}^k}\otimes \diago (D_{\nu \rho}^k)^T-\big(A_3 x_{\epsilon_N}^{\tilde{m}^k} + g_\lambda(x_{\epsilon_N}^{\tilde{m}^k}) \big)\otimes \diago(D_\lambda^k)^T -g_\rho(x_{\epsilon_N}^{\tilde{m}^k})\otimes \diago(D_\rho^k)^T \text{.}
\end{align*}
At the Newton step on the subset $\mathcal{I}_k$, we find $S_k\in \mathbb{R}^{M \times |\mathcal{I}_k|}$ such that
\begin{equation}\label{equation_matrix_fourpars1}
\begin{aligned}
A_0S_k+A_1S_kD_{\mu-}^k+A_2S_kD_{\nu\rho-}^k+A_3S_kD_{\lambda-}^k\\+J_\mu(x_{\epsilon_N}^{\tilde{m}^k})S_kD_{\mu}^k+J_\lambda(x_{\epsilon_N}^{\tilde{m}^k})S_kD_{\lambda}^k+J_\rho(x_{\epsilon_N}^{\tilde{m}^k})S_k D_{\rho}^k&=B_k^{\mu,\nu,\lambda,\rho}(x_{\epsilon_N}^{\tilde{m}^k}) \text{.}
\end{aligned}
\end{equation}
For a four-parameter discretization, the matrix equation (\ref{equation_matrixeq_algo1}) in \cref{algorithm_one1} is replaced by (\ref{equation_matrix_fourpars1}).


\section{Numerical study}\label{chap_numerical_example1}

\begin{figure}[t]
\centering
\includegraphics[width=0.9\textwidth]{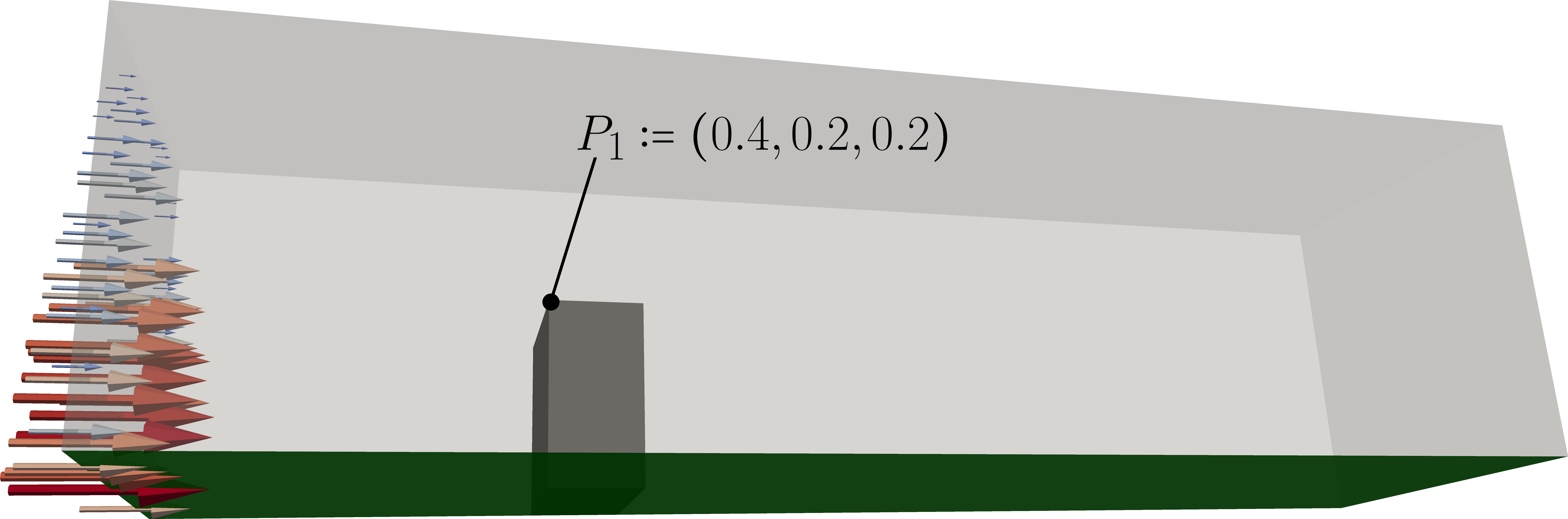}
\caption{Half of the domain of the initial configuration of the test case from \cite[Section 8.3.2.2]{Ric17}. The symmetry plane is in green. The evaluation point $P_1$ will be relevant in \cref{chapter_posteriori_error_estimate}.}
\label{figure_initial_geometry1}
\end{figure}

We consider the $3d$ test case from \cite[Section 8.3.2.2]{Ric17} with the domain $\Omega$, split into solid $S$ and fluid $F$
\begin{align*}
  \Omega :=(0,1.5) \times (0,0.4) \times (-0.4,0.4)\text{,} \qquad S:=(0.4,0.5)\times (0,0.2) \times (-0.2,0.2) \qquad \text{and}\qquad F:=\Omega \setminus \bar{S}\text{.}
\end{align*}
A parabolic inflow with an average velocity $v_{\text{in}}\approx 0.15$ is given at $x=0$. At $x=1.5$, the do-nothing condition $\nu_f\rho_f\nabla v n_F-pn_F=0$ holds for the velocity and the pressure. The problem is plane symmetric with respect to $z=0$. This is why we compute approximations in one half of the domain only (compare \cref{figure_initial_geometry1}). On the symmetry plane we demand
\[
v_f\cdot n=0 \qquad \text{and} \qquad u_s\cdot n =0 \qquad \text{for}\qquad z=0.
\]
On the right boundary at $x=1.5$ the do-nothing outflow condition holds and on all boundaries, velocity and deformation fulfill zero Dirichlet boundary conditions.

In this paper, we use 
tri-quadratic finite elements on a hexahedral mesh for the pressure, the velocity and the deformation. The Navier-Stokes equations are stabilized by local projection stabilization (LPS)~\cite{BeckerBraack2001} and \cite[Lemma 4.49]{Ric17}, by projecting the pressure onto the space of tri-linear finite elements on the same mesh, see \cite[Section 4.3.2]{Ric17}. Hereby the equal-order approach is applicable to pressure, deformation and velocity, which is relevant for formulating the matrix equation, compare~\eqref{equalorder}. Our finite element mesh $\Omega_h$ and discretization is such that the total number of unknowns is $M=66\,759$. 


The first  Lam\'e parameter $\lambda_s$ to is taken as $2\,000\,000$ and the fluid density $\rho_f$ as $\rho_f=1\,000$. The stationary fluid-structure interaction problem (\ref{equation_fsi_problem_weak1}) is discretized with  respect to the parameter set
\begin{align*}
  m_1&=500 \text{ shear moduli } \mu_s^{i_1} \in [\mbox{400\hspace*{.037cm}000},\mbox{600\hspace*{.037cm}000}] \text{ and} \\ m_2&=10 \text{ kinematic fluid viscosities } \nu_f^{i_2} \in [0.001,0.003]\text{.}
\end{align*}
With this choice we cover solids with Poisson ratios between $0.38$ and $0.41$ and Reynolds numbers between $\approx 20$ and $\approx 60$ if the characteristic length is assumed to be $L=0.4$.

All computations were performed with MATLAB\textsuperscript{\textregistered} 2018a in combination with the finite element toolkit \emph{Gascoigne 3d} \cite{software_gascoigne1} on a CentOS Linux release 7.5.1804 with a dual-socket Intel Xeon Silver 4110 and 192
GB RAM. \cite[Algorithm 6]{software_htucker1} was used to realize the truncation operator $\mathcal{T}(\cdot)$ from \cite[Definition 4]{WeiBR20} that maps, for $R \in \mathbb{N}$, into $T_R$, the space of Tucker tensors from \cite[Definition 3]{WeiBR20}. Using the MATLAB builtin command \texttt{lu()}, the preconditioners were decomposed into a permuted LU decomposition.

\label{chapter_numerical_comparison1}

\subsection{Low-rank approximation versus separate Newton iterations}\label{section_alg1_vs_sepnewton1}

\begin{figure}[t]
\centering
\subfloat[Results based on the ChebyshevT method]
         {\label{figure_alg1_chebyshevt1}
           \includegraphics[width=0.48\textwidth]{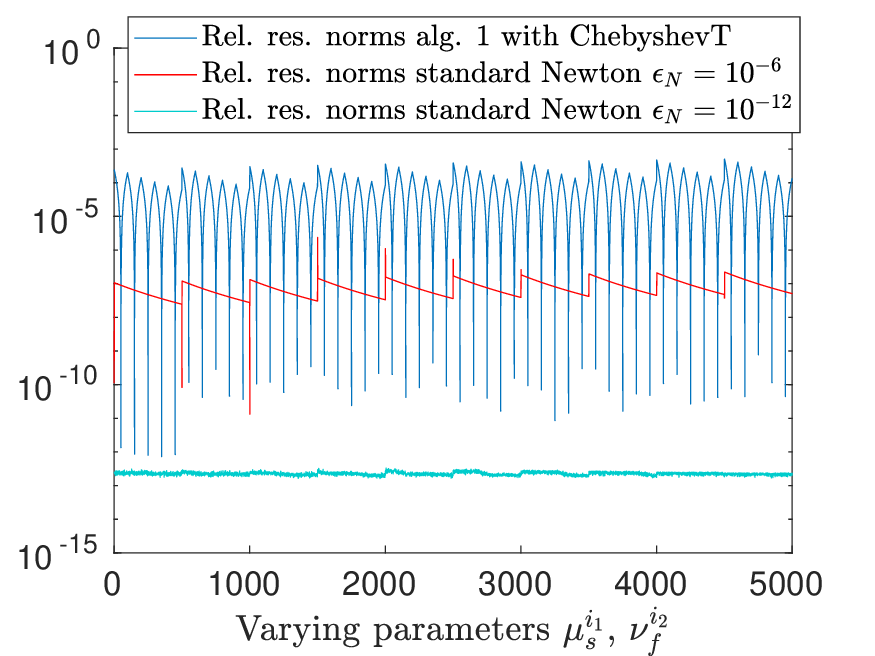}}
         \subfloat[Results based on the GMREST method]
         {\label{figure_alg1_gmrest1}\includegraphics[width=0.48\textwidth]{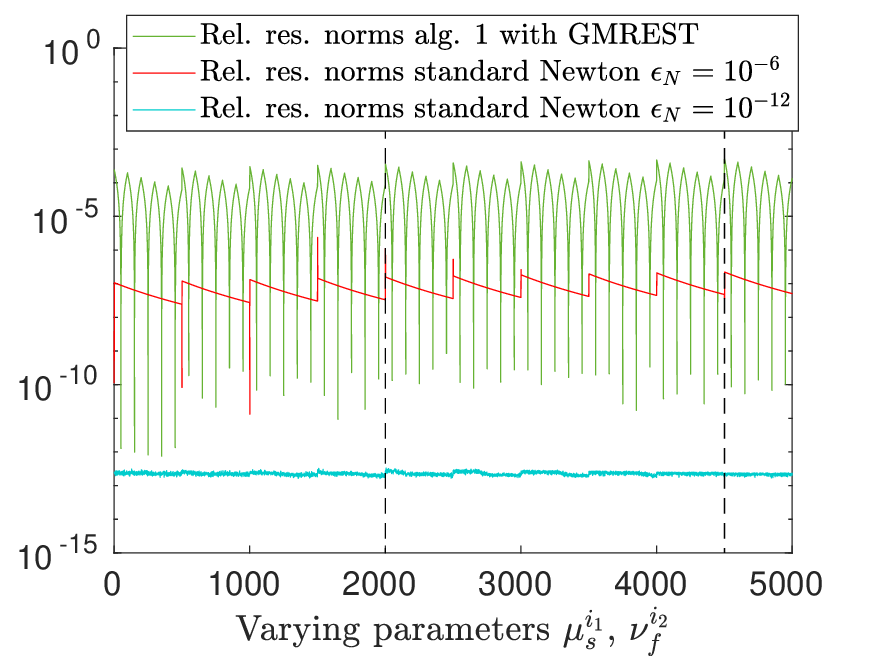}}

\caption{The norms of the relative residuals of the approximations computed by \cref{algorithm_one1} compared with the ones of the approximations computed by the standard Newton methods.}
\label{figure_alg1_compared_both}
\end{figure}

As a first test we compare the performance of \cref{algorithm_one1} with the individual solution of all problems by a Newton iteration.
In \cref{algorithm_one1} we set  $\epsilon_N=10^{-6}$, $K=50$ and $R_k=10$ for all $k \in \{1,...,K\}$. The indices of the parameter set $S^{\mu, \nu}$ were ordered little endian. The ChebyshevT and the GMREST methods both were restarted once after $6$ iterations. To compute $x_{\epsilon_N}^{\tilde{m}^k}$, for $k>1$, $x_{\epsilon_N}^{\tilde{m}^{k-1}}$ served as initial guess for the Newton iteration such that mostly a single Newton step is sufficient.

\subsubsection{ChebyshevT parameter estimation}To estimate the parameters $c$ and $d$ for the ChebyshevT method (compare \cite[Section 4.3]{WeiBR20}), a downgraded variant of \cref{algorithm_one1} was ran for a small number of $M_c=$ \mbox{$1$\hspace{.037cm}$575$} degrees of freedom. Let $\bar{\mathcal{I}}_k \subset \mathcal{I}_k$ such that
\begin{align*}
\bar{\mathcal{I}}_k=\{(\mu_s^{l_1},\nu_f^{l_2}) \in \mathcal{I}_k \text{ : } \big(\mu_s^{l_1} \geq \mu_s^{i_1} \vee \mu_s^{l_1} \leq \mu_s^{i_1}\big) \wedge \big( \nu_f^{l_2} \geq \nu_f^{i_2} \vee \nu_f^{l_2} \leq \nu_f^{i_2} \big) \quad \text{for all} \quad (\mu_s^{i_1},\nu_f^{i_2}) \in \mathcal{I}_k\}
\end{align*}
and define
\begin{align*}
Bl^k(\mu_s^{i_1},\nu_f^{i_2}):=(P_T^k)^{-1}A(x_{\epsilon_N}^{\tilde{m}^k}\mu_s^{i_1},\nu_f^{i_2})\text{,}
\end{align*}
with the preconditioner $P_T^k$ from \cref{chapter_preconditioner1}.
The quantities
\begin{align*}
\Lambda_{\text{max}}^k:=\max \{|\lambda| : \lambda \in \Lambda\big((\mathcal{P}_T^{k})^{-1}\mathcal{A}_k \big) \} \qquad \text{and} \qquad \Lambda_{\text{min}}^k :=\min \{ |\lambda|: \lambda \in \Lambda \big( (\mathcal{P}_T^{k})^{-1} \mathcal{A}_k 
\big) \}
\end{align*}
were approximated by
\begin{align*}
\bar{\Lambda}_\text{max}^k:=\max \limits_{(\mu_s^{i_1},\nu_f^{i_2})\in \bar{\mathcal{I}}_k} \{ |\lambda| : \lambda \in  \Lambda\big(Bl(\mu_s^{i_1},\nu_f^{i_2})\big) \}\qquad \text{and}\qquad \bar{\Lambda}_{\text{min}}^k:=\min \limits_{(\mu_s^{i_1},\nu_f^{i_2})\in \bar{\mathcal{I}}_k} \{ |\lambda| : \lambda \in \Lambda \big(Bl(\mu_s^{i_1},\nu_f^{i_2})\big)\}
\text{.}
\end{align*}
The Newton approximations $x_{\epsilon_N}^{\tilde{m}^k}$ in the for-loop of \cref{algorithm_one1} were computed to approximate the values $\bar{\Lambda}_{\text{max}}^k$ and $\bar{\Lambda}_{\text{min}}^k$ for all the subsets. Based on the smallest interval $[\bar{\Lambda}_{\text{min}}, \bar{\Lambda}_{\text{max}}] \subset \mathbb{R}$ such that
\begin{align*}
[\bar{\Lambda}_{\text{min}}^k, \bar{\Lambda}_{\text{max}}^k]\subset [\bar{\Lambda}_{\text{min}}, \bar{\Lambda}_{\text{max}}] \qquad \text{for all} \qquad k \in \{1,...,K\}\text{,}
\end{align*} the Chebyshev parameters $c$ and $d$ were chosen according to \cite[Section 4.3]{WeiBR20}. Estimating them to $c=0.1$ and $d=1$ took $19.8$ minutes.

\begin{table}[t]
  \centering
  \normalsize
  \begin{tabular}{|c|c|c|c|c|c|c|} \hline
    \bf Method & \bf Approx. & \bf Newton &  \multicolumn{4}{c|}{\textbf{Computation Times} (Minutes)}  \\ 
    &\bf Storage& \bf Steps &Est.&Indiv.&Comp.&\bf Total
    \\
    \hline
    \textbf{Alg. 1}, $R=500$ with& $O[(M+m+R)R]$ & 50+106& -& 815.5& 426.4&1241.9 \\
    GMREST (\textcolor{green_plot1}{green}) & $\approx 275.65$MB & =156& &  & & $\approx$ \textbf{21 hours}
    \\
    \hline
    \textbf{Alg. 1}, $R=500$ with& $O[(M+m+R)R]$ & 50+106&19.8 & 815.4&426 &1261.2 \\
    ChebyshevT (\textcolor{blue_plot1}{blue}) & $\approx 275.65$MB &=156  & &  & &$\approx$ \textbf{21 hours}\\
    \hline
    \textbf{Standard Newton}, $5000$& $O[Nm]$ &5016 &- &- &- &38757 \\
    times, $\epsilon_N=10^{-6}$ (\textcolor{red_plot1}{red}) & $\approx 2546.65$MB &  & &  & &$\approx$ \textbf{27 days}\\
    \hline
    \textbf{Standard Newton}, $5000$& $O[Nm]$ &10421 &- &- &- &80108 \\
    times, $\epsilon_N=10^{-12}$ (\textcolor{cyan_plot1}{cyan}) & $\approx 2546.65$MB &  & &  & &$\approx$ \textbf{56 days}\\
    \hline
  \end{tabular}%
  \caption*{\footnotesize$50$ Newton steps were required by Algorithm 1 on the subsets $\mathcal{I}_k$, $106$ to compute the approximations $x_{\epsilon_N}^{\tilde{m}^k}$. The column "Est." shows the time needed to estimate the ChebyshevT parameters $c$ and $d$, "Indiv." the time needed to compute the Newton approximations $x_{\epsilon_N}^{\tilde{m}_k}$ and "Comp." the time needed to approximate $S_k$ in the matrix equations and assemble $\hat{X}$.}
  \caption{\cref{algorithm_one1} Compared With the Standard Newton Method}
  \label{table_alg1_compared_stnewton1}
\end{table}

\subsubsection{The standard Newton method as a reference}
In comparison, the Newton iteration was applied to the $m=$ \mbox{$5$\hspace{.037cm}$000$} separate problems consecutively. The accuracy for the stopping criterion $\|residual\|_2 \leq \epsilon_N$ was once set to $\epsilon_N=10^{-6}$ and once to $\epsilon_N=10^{-12}$. For all but the first Newton iteration, the previous approximation served as initial guess. We call this the standard Newton method in this paper.

\par $b_D$ was used as initial guess to compute the Newton approximation $x_{\epsilon_N}^{\tilde{m}_1}$ in \cref{algorithm_one1} as well as for the first Newton iteration in the standard Newton method. The relative residual norms in \cref{figure_alg1_compared_both} are 
\begin{align*}
\frac{\|g(x^{i_1,i_2},\mu_s^{i_1},\nu_f^{i_2}) \|_2}{\|g(b_D,\mu_s^{i_1},\nu_f^{i_2})  \|_2} \text{,}
\end{align*}where $x^{i_1,i_2}$ is the approximation considered.
\par The relative residual norms of the approximations provided by the GMREST and the ChebyshevT method are all smaller than $10^{-3}$. The standard Newton method, in comparison, took about 27 days to compute approximations that provide relative residual norms that are all smaller than $10^{-6}$. The difference in the runtime of ChebyshevT and GMREST comes from the estimation of the parameters $c$ and $d$ for the ChebyshevT method. This poses the question whether the parameters for the ChebyshevT method, especially $c$, based on an initial guess, can be chosen adaptively depending on the convergence of the method.

\subsection{Error analysis}\label{chapter_posteriori_error_estimate}

We next explore the influence of the algebraic error arising from the approximation of the discrete equations. For this purpose, we consider two cases in detail. The relative residual norm of the approximation of the parameter combination $m=m_1m_2=$ \mbox{$4$\hspace{.037cm}$501$} (right dashed vertical line in \cref{figure_alg1_gmrest1}) is $\approx 5.1\cdot 10^{-4}$. Therefore, we would like to examine the quality of the approximation of the problem related to the parameter pair $(\mu_s^9,\nu_f^1)$ at $m=$ \mbox{$4$\hspace{.037cm}$501$} a bit more. In addition, we also consider the approximation of the problem related to $m=$ \mbox{$2$\hspace{.037cm}$001$}, the parameters $(\mu_s^4, \nu_f^1)$ (left dashed vertical line in \cref{figure_alg1_gmrest1}). There, the relative residual norm is $\approx 3.6\cdot 10^{-4}$.
For these two parameter combinations, each of which has a large error, we try to split the error into the discretization error and the error by applying the low-rank method.

To measure the quality of the approximation we consider the two error functionals that have been proposed in \cite{Richter2012a} and \cite[Section 8.3.2.2]{Ric17}. These are the $x$-deflection in the point $P_1=(0.4, 0.2, 0.2)$ from \cref{figure_initial_geometry1} and the force of the fluid onto the structure in main flow direction
\[
J_{x \text{-defl}}(x_h):=e_1\cdot u_h(P_1),\quad
J_{\text{drag}}(x_h):=\int \limits_{\Gamma_{\text{int}}} \sigma_f n_F \cdot e_1 ds.
\]
Here, $e_1 \in \mathbb{R}^3$ denotes the first unit vector and $u_h(P_1) \in \mathbb{R}^3$ is the deformation in the point $P_1$.

Let $(i_1,i_2) \in \Xi^{m_1, m_2}$. $\hat{x}^{i_1,i_2} \in \mathbb{R}^M$ denotes the approximation of the FSI problem related to the parameter combination $(\mu_s^{i_1},\nu_f^{i_2})$ provided by \cref{algorithm_one1} with the GMREST method. The respective approximation provided by the standard Newton method with $\epsilon_N=10^{-6}$ and $\epsilon_N=10^{-12}$ is denoted by $x^{{i_1},{i_2}}_{10^{-6}} \in \mathbb{R}^M$ and $x^{{i_1},{i_2}}_{10^{-12}} \in \mathbb{R}^M$, respectively. To compute a reference solution, we go to a finer grid, where the number of degrees of freedom is $M_f=$ \mbox{$495$\hspace{.037cm}$495$}. The Newton stopping criterion is set to $\|residual\|_2\leq 10^{-12}$ and we denote the approximation obtained by $x_{\frac{h}{2}}^{{i_1},{i_2}} \in \mathbb{R}^{M_f}$.

\begin{table}[t]
  \centering
  \normalsize
  \begin{tabular}{|c|c|c|c|c|c|} \hline
    \multicolumn{2}{|c|}{\textbf{Parameters}} & $\boldsymbol{\operatorname{err}_{x\textbf{-defl}}^{P_1} }$ &  $  \boldsymbol{ \operatorname{errdisc}_{x\textbf{-defl}}^{P_1}  } $ &  $\boldsymbol{\operatorname{err}_{\textbf{drag}} }$ & $ \boldsymbol{\operatorname{errdisc}_{\textbf{drag}} } $
    \\
    $\boldsymbol{(\mu_s^{i_1},\nu_f^{i_2})}$&$\boldsymbol{\epsilon_N}$ &&&&
    \\
    \hline
    \multirow{2}{*}{$(\mu_s^4,\nu_f^1)$}&$10^{-6}$& $2.626\cdot 10^{-7}$ & $4.566\cdot 10^{-3}$ &$1.043\cdot 10^{-7}$&$1.375\cdot 10^{-2}$	
    \\
    \cline{2-6}
    &$10^{-12}$&$2.655\cdot 10^{-7}$&$4.566\cdot 10^{-3}$&$1.046\cdot 10^{-7}$&$1.375\cdot 10^{-2}$	\\
    \hline
    \multirow{2}{*}{$(\mu_s^9,\nu_f^1)$}&$10^{-6}$&$4.415\cdot 10^{-7}$&$1.048\cdot 10^{-2}$&$1.200\cdot 10^{-7}$&$1.129\cdot 10^{-2}$
    \\
    \cline{2-6}
    &$10^{-12}$&$4.432\cdot 10^{-7}$&$1.048\cdot 10^{-2}$&$1.177\cdot 10^{-7}$&$1.129\cdot 10^{-2}$	\\
    \hline
  \end{tabular}%
  \caption{Error Analysis for \cref{algorithm_one1}}
\caption*{\footnotesize The discretization errors computed are smaller than the errors of \cref{algorithm_one1}.}
\label{table_posteriori1}
\end{table}

We list the relative quantities
\begin{align*}
  \operatorname{err}_{x\text{-defl}}^{P_1}&:= \frac{|J_{x\text{-defl}}(\hat{x}^{i_1,i_2})-J_{x\text{-defl}}(x_{\epsilon_N}^{i_1,i_2})|}{ |J_{x\text{-defl}}(x_{\frac{h}{2}}^{i_1,i_2})|} \text{,} \qquad \hspace{.732cm} \operatorname{errdisc}_{x\text{-defl}}^{P_1}:=\frac{|J_{x\text{-defl}}(x_{\frac{h}{2}}^{i_1,i_2})-J_{x\text{-defl}}(x_{\epsilon_N}^{i_1,i_2})|}{|J_{x\text{-defl}}(x_{\frac{h}{2}}^{i_1,i_2})|} \text{,}\\
  \operatorname{err}_{\text{drag}}&:= \frac{|J_{\text{drag}}(\hat{x}^{i_1,i_2})-J_{\text{drag}}(x_{\epsilon_N}^{i_1,i_2})|}{| J_{\text{drag}}(x_{\frac{h}{2}}^{i_1,i_2})|}  \qquad \text{and} \qquad
  \operatorname{errdisc}_{\text{drag}}:=\frac{|J_{\text{drag}}(x_{\frac{h}{2}}^{i_1,i_2})-J_{\text{drag}}(x_{\epsilon_N}^{i_1,i_2})|}{| J_{\text{drag}}(x_{\frac{h}{2}}^{i_1,i_2})|}\text{,}
\end{align*}
where $(i_1,i_2) \in \Xi^{m_1,m_2}$ and $\epsilon_N \in \{10^{-6}, 10^{-12}\}$ in \cref{table_posteriori1}. We assume that $x_{h/2}^{i_1,i_2}$ approximates the solution sufficiently well. Therefore, the quantities $\operatorname{errdisc}_{x\text{-defl}}^{P_1}$ and $\operatorname{errdisc}_{\text{drag}}$ can be considered as approximate values of the discretization errors. We estimate the relative error
\begin{align*}
  \frac{|J_{\text{drag}}^{i_1,i_2}-J_{\text{drag}}(\hat{x}^{i_1,i_2})|}{J_{\text{drag}}(x_{\frac{h}{2}}^{i_1,i_2}) } &\approx \frac{|J_{\text{drag}}(x_{\frac{h}{2}}^{i_1,i_2})-J_{\text{drag}}(\hat{x}^{i_1,i_2})|}{J_{\text{drag}}(x_{\frac{h}{2}}^{i_1,i_2}) }\\ &\leq \operatorname{errdisc}_{\text{drag}}+\operatorname{err}_{\text{drag}}\text{,}
\end{align*}
where $J_{\text{drag}}^{i_1,i_2}$ denotes the force of the fluid on the structure in $x$-direction of the solution. Furthermore, the error $\operatorname{err}_{\text{drag}}$ is the error of \cref{algorithm_one1}.
\par As listed in \cref{table_posteriori1}, we see that the discretization errors, $\operatorname{errdisc}_{x\text{-defl}}^{P_1}$ and $\operatorname{errdisc}_{\text{drag}}$, are always bigger than $\operatorname{err}_{x\text{-defl}}^{P_1}$ and $\operatorname{err}_{\text{drag}}$, the errors made by the GMREST method. In this sense, the GMREST method provided  approximations with accuracies that are at least of order $10^{4}$ more accurate than required.

\subsection{Fixed-point iteration and alternative clustering}\label{section_fixpt_and_alt_cluster}

\begin{figure}[t]
\centering
\subfloat[All relative residual norms in comparison]{\label{figure_picard1}\includegraphics[width=0.48\textwidth]{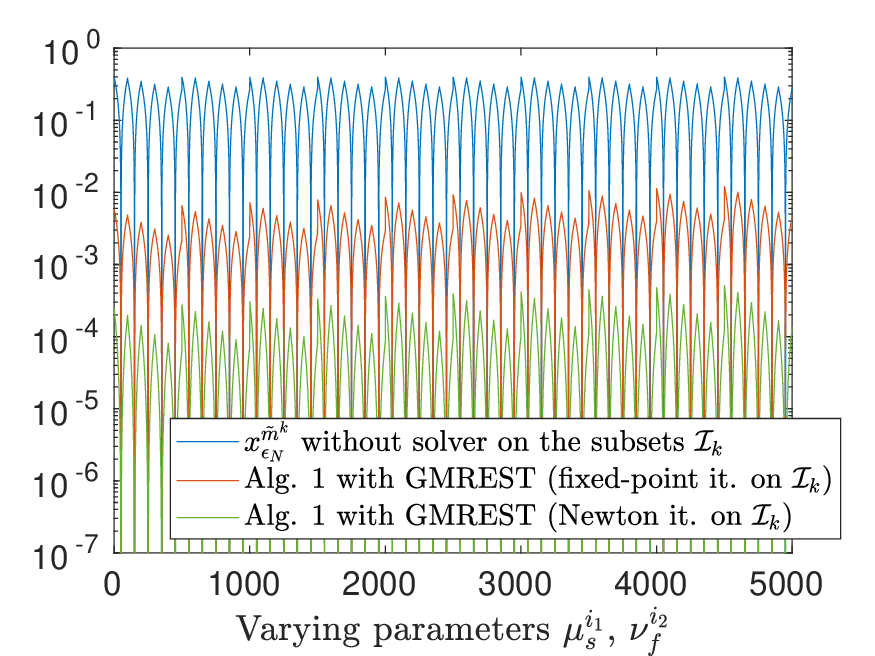}}
\subfloat[Relative residual norms of the problems related to parameter combinations $m=$ \mbox{$1$\hspace{.037cm}$500$} to $m=$ \mbox{$2$\hspace{.037cm}$500$}]{\label{figure_picard2}\includegraphics[width=0.48\textwidth]{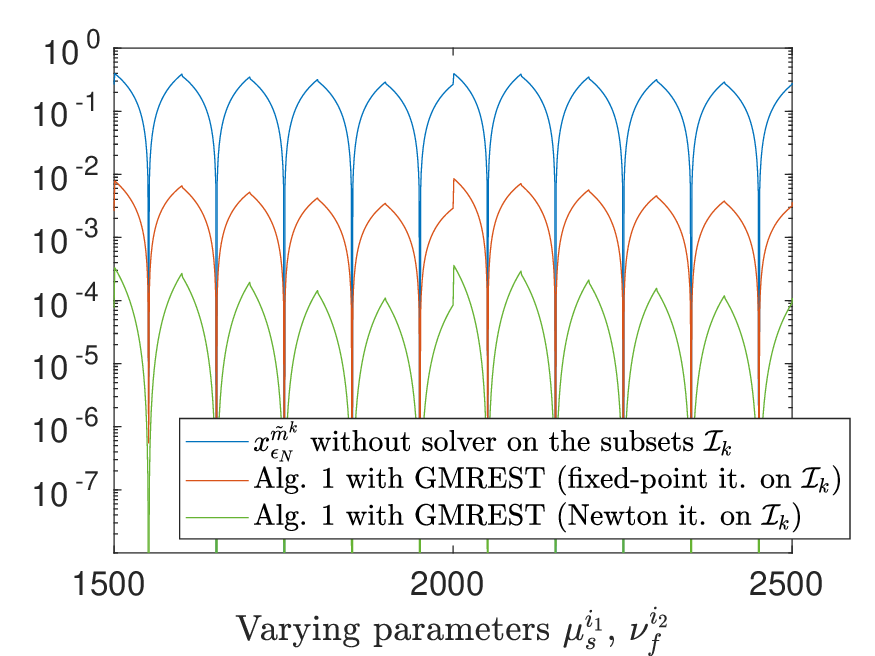}}
\caption{Newton versus fixed-point iteration on the subset level}
  \label{figure_picard_both}
\end{figure}

\cref{algorithm_one1} is based on the Newton iteration. We now investigate why this Newton approach is more advantageous  than basing \cref{algorithm_one1} on the Picard (fixed-point) iteration in the first place.

\label{chap_picard_iteration1}

An advantage of the fixed-point over the Newton iteration is that the right hand side does not need to be evaluated newly at every iteration. As in \cref{chapter_jacobianmatrices1}, we define the matrices $F_\rho(\cdot)$, $F_\mu(\cdot)$ and $F_\lambda(\cdot) \in \mathbb{R}^{M \times M}$ such that
\begin{align*}
F_\rho(x_h) \qquad &\text{discretizes} \qquad \langle  (v_h \cdot \nabla)v,\varphi \rangle_F \text{,} \\
F_\mu(x_h) \qquad &\text{discretizes} \qquad \langle \nabla u_h^T \nabla u, \nabla \varphi \rangle_S \qquad  \text{and}\\
F_\lambda(x_h) \qquad &\text{discretizes} \qquad \frac{1}{2}\langle \operatorname{tr}(\nabla u_h^T \nabla u)I, \nabla \varphi \rangle_S\text{.}
\end{align*}
$v_h$ is the velocity of the unknown $x_h$, the approximation of the previous linearization step. Since we solve the actual equation for $v$, replacing
\begin{align*}
(v \cdot \nabla )v \qquad \text{by}  \qquad ( v_h \cdot \nabla )v
\end{align*}
coincides with the Oseen fixed point linearization of the convection term mentioned in \cite[Section 4.4.1]{Ric17}. With the notation from  (\ref{equation_newton_matrix1}), a fixed-point step on the subset $\mathcal{I}_k$ translates to the problem of finding $X^k \in \mathbb{R}^{M \times |\mathcal{I}_k|}$ such that
\begin{equation}\label{equation_picard_matrix1}
\begin{aligned}
A_0X^k+A_1X^kD_{\mu-}^k+\rho_fA_2X^kD_{\nu-}^k+F_\mu(x_{\epsilon_N}^{\tilde{m}^k})X^kD_\mu^k\\+\lambda_s F_\lambda(x_{\epsilon_N}^{\tilde{m}^k})X^k+\rho_f F_\rho(x_{\epsilon_N}^{\tilde{m}^k})X^k&=b_D\otimes (1,...,1)\text{.}
\end{aligned}
\end{equation}
In \cref{algorithm_one1}, we replace (\ref{equation_matrixeq_algo1}) by (\ref{equation_picard_matrix1}). Once $x_{\epsilon_N}^{\tilde{m}^k}$ is computed on the respective subset, we directly approximate $X^k$  in (\ref{equation_matrixeq_algo1}) with a low-rank method. No update is needed here.

The variant of \cref{algorithm_one1} that uses fixed-point iteration on the subset level $\big($Alg. 1 with GMREST (fixed-point it. on $\mathcal{I}_k$)$\big)$ is compared with \cref{algorithm_one1} in \cref{figure_picard_both}. The relative residual norms (in \textcolor{green_plot1}{green}) denoted by Alg. 1 with GMREST (Newton it. on $\mathcal{I}_k$) in \cref{figure_picard_both} coincide with the ones in \cref{figure_alg1_gmrest1}. For \cref{algorithm_one1} based on fixed-point iteration (in \textcolor{red_plot1}{red}), we used not only the identical setup as the one in \cref{chapter_numerical_comparison1} but also exactly the same initial guesses $x_{\epsilon}^{\tilde{m}^k}$ for $k \in \{1,...,K\}$. In addition, we plot the relative residual norms of $x_{\epsilon_N}^{\tilde{m}^k}$ on the respective subsets $\mathcal{I}_k$ (in \textcolor{blue_plot1}{blue}) for visualization. 
\begin{remark}
At the parameter combination corresponding to $m=$ \mbox{$2$\hspace{.037cm}$000$}, another initial guess is used than at the parameter combination corresponding to $m=$ \mbox{$2$\hspace{.037cm}$001$}. This is why there is a kink at $m=$ \mbox{$2$\hspace{.037cm}$000$}.
\end{remark}
\par A Newton iteration on the subset level provides faster convergence and therefore, more accurate approximations than a fixed-point iteration on the subset level (compare \cref{figure_picard1}). This is illustrated in \cref{figure_picard2}, where the approximations computed by the algorithm based on the Newton iteration has a relative residual norm that is at least $10$ times smaller than the ones computed by the algorithm based on the fixed-point iteration.
 \subsection{Alternative ways to cluster the parameter set}\label{chap_alternative_clustering1}
 \begin{figure}[tbhp]
	\centering
  \includegraphics[scale=0.6]{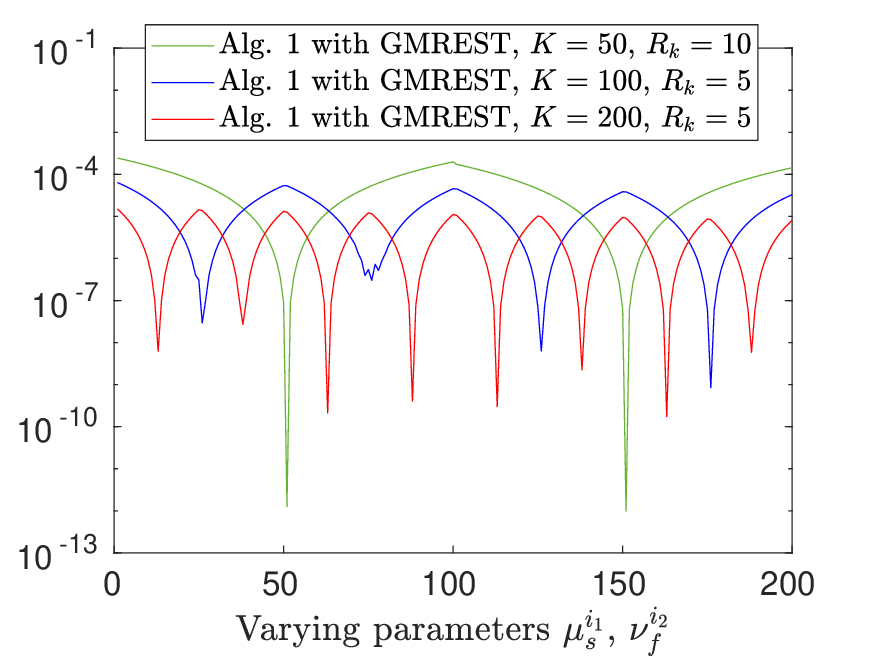}
	\caption{The relative residual norms of the approximations obtained by \cref{algorithm_one1} with different numbers of subsets.}
	\label{figure_adapt1}
\end{figure}
In \cite{WeiBR19a}, a one-parameter discretization of a nonlinear FSI problem was performed. The solid equations used there are linear and the residual norms in the respective numerical example obtained by \cite[Algorithm 1]{WeiBR19a} are all smaller than $10^{-9}$. This is remarkably smaller than the ones obtained in \cref{chap_numerical_example1}. The reason for this is, on the one hand, that the kinematic fluid viscosity in the numerical example in \cite{WeiBR19a} is, $\nu_f=0.04$, chosen much bigger than the kinematic fluid viscosities considered in this paper. On the other hand, the equations on the solid of the problem considered in this paper are nonlinear. The problems that are grouped by a subset $\mathcal{I}_k$ differ such that a Newton step in the form (\ref{equation_newton_matrix1}) converges for problems that do not lie central in $\mathcal{I}_k$, in the worst case, to a relative residual norm not smaller than $\approx 5.1 \cdot 10^{-4}$. The question whether choosing the subsets adaptively would bring any benefits arises in this context.
 
In \cref{figure_adapt1}, we compare the relative residual norms of the approximations obtained by \cref{algorithm_one1} with GMREST for different numbers of subsets. The \textcolor{green_plot1}{green} line corresponds to the ones of the first $200$ parameter combinations in \cref{figure_alg1_gmrest1}. For $K=100$ (corresponding line in \textcolor{blue_plot1}{blue}), the maximal relative residual norm of the approximations to the problems related to the first $200$ parameter combinations could be reduced from $\approx 2.44\cdot 10^{-4}$ to $\approx 6.34\cdot 10^{-5}$. A further refinement to $K=200$ (line in \textcolor{red_plot1}{red}) resulted in a maximal relative residual norm of $\approx 1.51\cdot 10^{-5}$. Reducing the rank $R_k$ to $5$ (in case of $K=100$ and $K=200$) did not lead to an increasing maximum in terms of relative residual norms.
\par However, in multi-parameter discretizations like the one considered in \cref{chap_numerical_example1}, smaller subsets could help to get higher accuracies for some problems. But it is not clear how to group the problems if fluid and solid parameters are varied at the same time and the cardinalities of the subsets $\mathcal{I}_k$ are not similar for all $k \in \{1,...,K\}$. It is still open how the clusters \cref{algorithm_one1} uses can be chosen more sophisticatedly or adaptively refined in a multi-parameter setup that requires more accurate approximations.

\section{Conclusions}
\cref{algorithm_one1} allows to compute low-rank approximations for parameter-dependent nonlinear FSI problems with lower complexity than standard approaches. We reduce the complexity of the resulting algorithm by applying low-rank approaches. Even though we only perform one Newton step on the subset level, the errors of the approximations we obtain with \cref{algorithm_one1}, if compared with the ones provided by full approaches on the respective mesh, are smaller than the discretization errors. Essentially, \cref{algorithm_one1} is applicable to other coupled nonlinear problems such as FSI problems in ALE formulation and can be used, in principle, for finite difference, finite volume or any finite element discretization.  The approximations we obtained in \cref{chap_numerical_example1} provide accuracies that are more exact than required (compare \cref{table_posteriori1}). Moreover, we have demonstrated, why, on the subset level, the Newton linearization is to be preferred to fixed-point iteration. For multi-parameter discretizations, yet, it is not clear how to improve the choice of the subsets. The problems grouped by one subset should converge to a small residual norm in one Newton step with the same initial guess.

\section*{Acknowledgements}This work was supported by the Deutsche Forschungsgemeinschaft (DFG, German Research Foundation) - 314838170, GRK 2297 MathCoRe.





\bibliographystyle{siam}

\end{document}

%% file: introduction.tex
Fluid-structure interactions (FSI) play a crucial role in various applications, ranging from classical problems of aeroelasticity to naval engineering, biology and medicine, compare~\cite{BazilevsTakizawaTezduyar2013,Ric17} for examples. 
Simulations of fluid-structrure interactions serve to supplement or replace experiments. The two characteristic difficulties of fluid-structure interactions are the very stiff coupling arising from the interplay of a parabolic (fluid) and a hyperbolic equation (structure), and the moving and a priori unknown solution domain resulting from the motion and deformation of the structure. Finally, a nonlinear, poorly structured system of partial differential equations must be solved. 

Two approaches are generally considered to approximate the coupled
system. First, partitioned methods, which are based on an external
coupling of a flow model and a structural model. These methods are
highly flexible and efficient discretizations and solvers can be
considered for the particular subproblems. However, partitioned
methods suffer from stability problems, especially when problems with
pronounced added-mass effect are
considered~\cite{CausinGerbeauNobile2005}. This effect emerges when
the densities of fluid and structure are similar, for example in
shipbuilding applications or in biomedical problems. The alternative
monolithic approach considers the entire system as a single entity,
generally discretized using implicit methods. It results in a highly
robust description, which, however, leads to algebraic systems of
equations of great complexity, whose solution (in the nonlinear as
well as in the linear case) is extremely costly, compare~\cite{HronTurekMadlikEtAl2010,FailerRichter2019}. For a comparison of both approaches we refer to Heil et al.~\cite{HeilHazelBoyle2008}. Here we follow the monolithic approach, also because it allows us to formulate the parameter-dependent problem in a single unit.

The other problem, i.e., the motion of the domain, is usually addressed either by a transformation into a common coordinate system, or, especially in partitioned approaches, the coupling condition is realized by interpolation between the two subproblems. 
However, if a common coordinate system is chosen as the basis of the
description, on the one hand an Eulerian description is
suitable~\cite{Dunne2006,CottetMaitreMilcent2006,Richter2013,Pironneau2016},
i.e., both the flow and structure problem are described on the changing Eulerian domain, or on the other hand a fixed reference domain is chosen for both subproblems. This second framework is called the Arbitrary Lagrangian Eulerian (ALE)  coordinate system~\cite{Donea1982,HronTurek2006a,RichterWick2010}, which is based on a transformation of the Navier-Stokes equations into a fixed reference system.
We take a simple route here and assume that the deformation of the fluid domain is small, and thus the geometric nonlinearity does not play a crucial role, so the ALE transformation can be considered as an identity. In the elasticity model, on the other hand, we will consider the usual nonlinear strain tensor.

Here we target parameter-dependent fluid-structure interaction
problems in order to use the solution later in an optimization or
design process. Alternatively, an optimization problem with
fluid-structure interactions as constraints could be considered
directly. This approach has already been described in the
literature~\cite{Bertoglio12,RichterWickOpt,FTOpt2020}, but it remains
very complex due to the great complexity of the coupled system. For
technically realistic scenarios, no efficient and at the same time
accurate solution exists so far. Therefore, in this work we develop a
novel low-rank technique that solves the parametric problem using a
low-rank format in an all-at-once approach. For this purpose, we follow the ideas laid
out for the linear case in \cite{WeiBR20} and embed this in a
Newton iteration for the nonlinear problem considered here. 

We state the weak formulation of the nonlinear FSI problem in \cref{chap_weak_form1}. In \cref{chap_pardep_newton1}, the parameter-dependent discretization is linearized by means of the Newton iteration, and a Newton step is, for a subset of problems, translated to a matrix equation. We then discuss the approximability of the unknown, a matrix, by low-rank methods in \cref{chap_lrank_meth_and_gen1}, before we derive and code the novel algorithm in \cref{algorithm_one1}. The $3d$ FSI test case from \cite[Section 8.3.2.2]{Ric17} is considered for a numerical comparison of \cref{algorithm_one1} with standard Newton iterations in \cref{chap_numerical_example1}. Variants of \cref{algorithm_one1} are discussed in \cref{section_fixpt_and_alt_cluster}.